\documentclass[11pt,reqno]{amsart}
\usepackage{amsmath, amsfonts, amsthm, amssymb, graphicx}

\theoremstyle{plain}
\newtheorem{theorem}{Theorem}[section]
\newtheorem{lemma}{Lemma}[section]
\newtheorem{proposition}{Proposition}[section]
\theoremstyle{definition}
\newtheorem{definition}{Definition}[section]
\theoremstyle{remark}
\newtheorem{remark}{Remark}[section]
\theoremstyle{example}

\renewcommand{\u}{{ u}}            

\renewcommand{\xi}{{s}}

\newcommand{\e}{\varepsilon}          

\newcommand{\la}{\langle}             
\newcommand{\ra}{\rangle}          

\newcommand{\R}{\mathbb{R}}           
\renewcommand{\H}{\mathbb{H}}         

\newcommand{\supp}{\mbox{supp}}
\newcommand{\rhob}{\bar{\rho}}
\newcommand{\HH}{\mathcal{H}}
\newcommand{\SPH}{\mathbb{S}^1}
\newcommand{\Prob}{\,{\rm Prob}\,}
\newcommand{\sh}{{\sharp}}

\def\be{\begin{equation}}
\def\ee{\end{equation}}
\def\bes{\begin{equation*}}
\def\ees{\end{equation*}}
\def\bc{\begin{cases}}
\def\ec{\end{cases}}

\numberwithin{equation}{section}

\begin{document}
\title[Vanishing Viscosity Limit of the Navier-Stokes Equations]
{Vanishing Viscosity Limit of the Navier-Stokes Equations to the
Euler Equations for Compressible Fluid Flow}
\author{Gui-Qiang Chen \and Mikhail Perepelitsa}
\address{Gui-Qiang G. Chen, Mathematical Institute, University of Oxford,
         Oxford, OX1 3LB, UK; and Department of Mathematics, Northwestern University,
         Evanston, IL 60208, USA}
\email{\tt gqchen@math.northwestern.edu}
\address{Mikhail Perepelitsa,
Department of Mathematics, University of Houston, 651 PGH, Houston,
Texas 77204-3008; and Department of Mathematics, Vanderbilt
University, Nashville, TN 37240, USA.} \email{\tt misha@math.uh.edu}
\date{\today}
\thanks{}

\begin{abstract}
We establish the vanishing viscosity limit of the Navier-Stokes
equations to the isentropic Euler equations for one-dimensional
compressible fluid flow. For the Navier-Stokes equations, there
exist no natural invariant regions for the equations with the real
physical viscosity term so that the uniform sup-norm of solutions
with respect to the physical viscosity coefficient may not be
directly controllable and, furthermore, convex entropy-entropy flux
pairs may not produce signed entropy dissipation measures. To
overcome these difficulties, we first develop uniform energy-type
estimates with respect to the viscosity coefficient for the
solutions of the Navier-Stokes equations and establish the existence
of measure-valued solutions of the isentropic Euler equations
generated by the Navier-Stokes equations. Based on the uniform
energy-type estimates and the features of the isentropic Euler
equations, we establish that the entropy dissipation measures of
the solutions of the Navier-Stokes equations for weak
entropy-entropy flux pairs, generated by compactly supported $C^2$
test functions, are confined in a compact set in $H^{-1}$, which
lead to the existence of measure-valued solutions that are confined
by the Tartar-Murat commutator relation.
A careful characterization of the unbounded support of the
measure-valued solution confined by the commutator relation yields
the reduction of the measure-valued solution to a Delta mass, which
leads to the convergence of solutions of the Navier-Stokes equations
to a finite-energy entropy solution of the isentropic Euler
equations.
\end{abstract}
\maketitle

\section{Introduction}

We are concerned with the vanishing viscosity limit of the motion of
a compressible viscous, barotropic fluid in Eulerian coordinates
$\mathbb{R}_+^2:=[0, \infty)\times\mathbb{R}$, which is described by
the system of Navier-Stokes equations:
\begin{equation}
\label{Eq:NS-1}
\begin{cases}
\rho_t{}+{}(\rho\u)_x{}={}0,\\
(\rho\u)_t{}+{}(\rho\u^2{}+{}p)_x{}={}\e\u_{xx}{},\\
\end{cases}
\end{equation}
with the initial conditions:
\begin{equation}
\label{Eq:Initial-Conditions} \rho(0,x){}={}\rho_0(x),\quad
\u(0,x){}={}\u_0(x)
\end{equation}
such that $\lim_{x\to\pm\infty}(\rho_0(x), u_0(x))=(\rho^\pm,
u^\pm)$, where $\rho$ denotes the density, $u$ represents the
velocity of the fluid when $\rho>0$, $p$ is the pressure, $m=\rho u$
is the momentum, and $(\rho^\pm, u^\pm)$ are constant states with
$\rho^\pm>0$.
The physical viscosity coefficient $\e$ is restricted to $\e\in (0,
\e_0]$ for some fixed $\e_0>0$. The pressure $p$ is a function of
the density through the internal energy $e(\rho)$:
$$
p(\rho) =\rho\, e '(\rho)-e(\rho) \qquad \mbox{for}\,\, \rho\ge 0.
$$
In particular, for a polytropic perfect gas,
\begin{equation}\label{Eq:Pressure}
p(\rho)=\kappa\rho^\gamma, \qquad
e(\rho)=\frac{\kappa}{\gamma-1}\rho^\gamma,
\end{equation}
where $\gamma>1$ is the adiabatic exponent and, by the scaling, the
constant $\kappa$ in the pressure-density relation may be chosen as
$\kappa=\frac{(\gamma-1)^2}{4\gamma}$ without loss of generality.
One of the fundamental features of this system is that strict
hyperbolicity fails when $\rho\to 0$.

The vanishing artificial/numerical viscosity limit to the isentropic
Euler equations with general $L^\infty$ initial data has been
studied by DiPerna \cite{DiPerna}, Chen \cite{Chen,Chen2}, Ding
\cite{Ding}, Ding-Chen-Luo \cite{DCL}, Lions-Perthame-Souganidis
\cite{LPS}, and Lions-Perthame-Tadmor \cite{LPT}  via the method of
compensated compactness. Also see DiPerna \cite{DiPerna-g}, Morawetz
\cite{Mor}, Perthame-Tzavaras \cite{PerthameTzavaras}, and Serre
\cite{Serre-g} for the vanishing artificial/numerical viscosity
limit to general $2\times 2$ strictly hyperbolic systems of
conservation laws. The vanishing artificial viscosity limit to
general strictly hyperbolic systems of conservation laws with
general small $BV$ initial data was first established by
Bianchini-Bressan \cite{BB} via direct BV estimates with small
oscillation.
Also see LeFloch-Westdickenberg \cite{LW} for the existence of
finite-energy solutions to the isentropic Euler equations with
finite-energy initial data for the case $1<\gamma\le 5/3$.

The idea of regarding inviscid gases as viscous gases with vanishing
real physical viscosity can date back the seminal paper by Stokes
\cite{Stokes} and the important contribution of Rankine
\cite{Rankine}, Hugoniot \cite{Hugoniot}, and Rayleigh
\cite{Rayleigh} (cf. Dafermos \cite{Dafermos}). However, the first
rigorous convergence analysis of vanishing physical viscosity from
the Navier-Stokes equations \eqref{Eq:NS-1} to the isentropic Euler
equations was made by Gilbarg \cite{Gi} in 1951, when he established
the mathematical existence and vanishing viscous limit of the
Navier-Stokes shock layers.
For the convergence analysis confined in the framework of piecewise
smooth solutions; see Hoff-Liu \cite{HL},
G\`{u}es-M\'{e}tivier-Williams-Zumbrun \cite{GMWZ}, and the
references cited therein. The convergence of vanishing physical
viscosity with general initial data was first studied by
Serre-Shearer \cite{SS} for a $2\times 2$ system in nonlinear
elasticity with severe growth conditions on the nonlinear function
in the system.

In this paper,
we first develop new uniform estimates with respect to the real
physical viscosity coefficient for the solutions of the
Navier-Stokes equations with the finite-energy initial data and
establish the $H^{-1}$-compactness of weak entropy dissipation
measures of the solutions of the Navier-Stokes equations for any
weak entropy-entropy flux pairs generated by compactly supported
$C^2$ test functions. With these, the existence of measure-valued
solutions with possibly unbounded support is established, which are
confined by the Tartar-Murat commutator relation with respect to two
pairs of weak entropy-entropy flux kernels. Then we establish the
reduction of measure-valued solutions with unbounded support for the
case $\gamma\ge 3$ and, as corollary, we obtain the existence of
global finite-energy entropy solutions of the Euler equations with
general initial data for $\gamma\ge 3$. We further simplify the
reduction proof of measure-valued solutions with unbounded support
for the case $1<\gamma\le 5/3$ in LeFloch-Westdickenberg \cite{LW}
and extend to the whole interval $1<\gamma<3$ . With all of these,
we establish the first convergence result for the vanishing physical
viscosity limit of solutions of the Navier-Stokes equations to a
finite-energy entropy solution of the isentropic Euler equations
with finite-energy initial data. We remark that, combining
Propositions 6.2 and 7.2 in this paper with the uniform estimates in
\cite{LW}, we obtain the existence of finite-energy solutions to the
isentropic Euler equations with geometric effects
for the case $\gamma>5/3$, which is also supplement to the existence
result in \cite{LW} for $1<\gamma\le 5/3$.

\medskip
The organization of this paper is as follows. In Section 2, we
analyze some basic properties of weak entropy-entropy flux pairs in
the unbounded phase plane and introduce the notion of finite-energy
entropy solutions. In Section 3, we make several uniform estimates
for the solutions of the Navier-Stokes equations which are
independent of the real physical viscosity coefficient $\e>0$. These
estimates are essential for establishing the convergence of
vanishing viscosity limit of the solutions of the Cauchy problem
\eqref{Eq:NS-1}--\eqref{Eq:Initial-Conditions} for the Navier-Stokes
equations. In Section 4, we establish the $H^{-1}$--compactness of
entropy dissipation measures for solutions of
\eqref{Eq:NS-1}--\eqref{Eq:Initial-Conditions} with initial data
\eqref{Eq:Initial-Conditions} for any weak entropy-entropy flux
pairs generated by compactly supported $C^2$ test functions. In
Section 5, we employ the estimates in Sections 3--4 to construct the
measure-valued solutions with possibly unbounded support determined
by the solutions of the Navier-Stokes equations \eqref{Eq:NS-1} with
initial data \eqref{Eq:Initial-Conditions} and show that the
measure-valued solutions are confined by the Tartar-Murat commutator
relation for any two pairs of weak entropy-entropy flux kernels. In
Sections 6--7, we prove that any connected component of the support
of the measure-valued solutions must be bounded when $\gamma>1$,
which reduces to the case for the measure-valued solutions with
bounded support. Finally, in Section 8, we conclude the strong
convergence of vanishing viscosity limit of solutions of the
Navier-Stokes equations to a finite-energy entropy solution of the
isentropic Euler equations.

\section{Entropy for the Isentropic Euler Equations}

In this section we analyze some basic properties of weak entropy
pairs in the unbounded phase plane and introduce the notion of
finite-energy entropy solutions of the isentropic Euler equations
with the form:
\begin{equation} \label{E-1}
\begin{cases}
{}\rho_t{}+{} (\rho u)_x{}={}0,\\
{}(\rho u)_t{}+{}(\rho u^2+{}p)_x{}={}0.\\
\end{cases}
\end{equation}
System \eqref{E-1} is an archetype of nonlinear hyperbolic systems
of conservation laws:
$$
U_t + F(U)_x =0.  \label{1.3}
$$
For our case, $U=(\rho, m)^\top$ and $F(U)=(m,
\frac{m^2}{\rho}+p)^\top$ for $m=\rho u$.

For $\gamma>1$, the eigenvalues of system \eqref{E-1} are
\begin{equation}
\lambda_j =u + (-1)^j\theta \rho^\theta, \qquad j=1,2, \label{7.5}
\end{equation}
and the Riemann invariants are
\begin{equation}
w_j =u+(-1)^{j-1}\rho^\theta, \qquad j=1,2, \label{7.8}
\end{equation}
where $\theta=\frac{\gamma-1}{2}$. {}From \eqref{7.5},
$$
\lambda_2-\lambda_1=2\theta\rho^\theta \to 0 \qquad\mbox{as}\,\,
\rho\to 0.
$$
Therefore, system \eqref{E-1} is strictly hyperbolic when $\rho>0$.
However, near the vacuum $\rho=0$, the two characteristic speeds of
\eqref{E-1} may coincide and the system be nonstrictly hyperbolic.

\medskip
A pair of mappings $(\eta, q): \R_+^2:=\R_+\times\R\to \R^2$ is
called an entropy-entropy flux pair (or entropy pair for short) of
system \eqref{E-1} if $(\eta, q)$ satisfy the $2\times 2$ hyperbolic
system:
\begin{equation}
\nabla q(U)=\nabla\eta(U)\nabla F(U). \label{7.1.1.1}
\end{equation}
Furthermore, $\eta(\rho,m)$ is called a weak entropy if
\begin{equation}
\eta\Big|_{\begin{subarray}{l}
                \rho=0\\
                u=m/\rho \,\,\text{fixed}
        \end{subarray}}=0.
\label{7.1.1.2}
\end{equation}
An entropy pair is said convex if the Hessian $\nabla^2\eta(\rho,m)$
is nonnegative
in the region under consideration.

For example, the mechanical energy (a sum of the kinetic and
internal energy) and the mechanical energy flux
\begin{equation}
\eta^*(\rho,m)=\frac{1}{2}\frac{m^2}{\rho}+ e(\rho), \qquad
q^*(\rho,m)=\frac{1}{2}\frac{m^3}{\rho^2}+m e'(\rho)
\label{mech-energy}
\end{equation}
form a special entropy pair; $\eta^*(\rho,m)$ is convex for any
$\gamma>1$
in the region $\rho\ge 0$.

\medskip
Let $(\bar{\rho}(x),\bar{u}(x))$ be a pair of smooth monotone
functions satisfying $(\bar{\rho}(x), \bar{u}(x))=(\rho^\pm, u^\pm)$
when $\pm x\ge L_0$ for some large $L_0>0$.
The total mechanical energy for \eqref{Eq:NS-1} in
$\R$ with respect to the pair $(\bar{\rho},\bar{u})$ is
\begin{equation}
\label{def:energy} E[\rho, u](t):=\int_{\R}\Big(\eta^*(\rho,
m)-\eta^*(\bar{\rho}, \bar{m})-\nabla\eta^*(\bar{\rho},
\bar{m})\cdot (\rho-\bar{\rho}, m-\bar{m})\Big)\, dx\ge 0,
\end{equation}
where $\bar{m}=\bar{\rho}\bar{u}$.

In the coordinates $(\rho,u)$, any weak entropy function
$\eta(\rho,\rho u)$ is governed by the second-order linear wave
equation:
\begin{equation}
\begin{cases}
{}\eta_{\rho\rho}-\frac{p'(\rho)}{\rho^2}\eta_{uu}=0, \qquad \rho>0,\\
{}\eta|_{\rho=0}=0.
\end{cases}
\label{entropy:1}
\end{equation}
Therefore, any weak entropy pair $(\eta,q)$ can be represented by
\begin{equation}
\begin{cases}
{}\eta^\psi(\rho,\rho u)=\int_{\R}\chi(\rho; s-u)\psi(s)\,ds,\medskip\\
{}q^\psi(\rho,\rho u)=\int_{\R}\big(\theta s+(1-\theta)u)\chi(\rho;
s-u)\psi(s)\,ds
\end{cases}
\label{entropy:3}
\end{equation}
for any continuous function $\psi(s)$, where the weak entropy kernel
$\chi(\rho, s-u)$ is determined by
\begin{equation}
\begin{cases}
{}\chi_{\rho\rho}-\frac{p'(\rho)}{\rho^2}\chi_{uu}=0,\\
{}\chi(0,u;s)=0,\quad \chi_\rho(0,u;s)=\delta_{u=s},
\end{cases}
\label{kernel:1}
\end{equation}
where $\delta_{u=s}$ is the Dirac mass concentrated at $u=s$.

This implies that, for the $\gamma$-law case, the weak entropy
kernel as the unique solution of \eqref{kernel:1} is
\begin{equation}\label{kernel}
\chi(\rho; s-u)=[\rho^{2\theta}-(s-u)^2]_+^\lambda,
\end{equation}
where $\lambda=\frac{3-\gamma}{2(\gamma-1)}>-\frac{1}{2}$. Then the
weak entropy pairs have the form:
\begin{eqnarray}
\eta^\psi(\rho,m)=\eta^\psi(\rho,\rho\u)
&=&\int_{\mathbb{R}}[\rho^{2\theta}{}-{}(\u-s)^2\big]_+^\lambda\psi(s)\,ds
  \nonumber\\
&=&\rho\int_{-1}^1\psi(u+\rho^\theta s)[1-s^2]^\lambda_+\,ds,
\label{eta}
\end{eqnarray}
\begin{eqnarray}
q^\psi(\rho,m)=q^\psi(\rho,\rho\u)&=&\int_{\mathbb{R}}(\theta
s{}+{}(1-\theta)\u)[\rho^{2\theta}{}-{}(\u-s)^2]_+^\lambda
\psi(s)\,ds\nonumber\\
&=&\rho\int_{-1}^1(u+\theta\rho^\theta s)\psi(u+\rho^\theta
s)[1-s^2]_+^\lambda ds. \label{q}
\end{eqnarray}

In particular, when $\psi_\sh(w)=\frac{1}{2}w|w|$,  the
corresponding entropy pair $(\eta^\sh, q^\sh):=(\eta^{\psi_\sh},
q^{\psi_\sh})$ satisfies that there exists $C>0$, depending only on
$\gamma
>1$, such that
\begin{equation}\label{2.0}
|\eta^\sh(\rho,m)|{}\leq
C\big(\rho|\u|^2{}+{}\rho^\gamma\big),\qquad
q^\sh(\rho,m){}\geq{}C^{-1}\big(\rho|\u|^3{}+{}\rho^{\gamma+\theta}\big),
\end{equation}
\begin{equation}\label{2.1a}
|\eta^{\sh}_m(\rho, m)|\le C(|u|+\rho^\theta), \qquad
|\eta^\sh_{mm}(\rho, m)|\le C\rho^{-1},
\end{equation}
and, regarding $\eta^\sh_m$ in the coordinates $(\rho, u)$,
\begin{equation}\label{2.1b}
|\eta^{\sh}_{mu}(\rho, \rho u)|\le C, \qquad
|\eta^{\sh}_{m\rho}(\rho, \rho u)|\le C\rho^{\theta-1}
\end{equation}
for all $\rho\ge 0$ and $u\in\R$ (also see, e.g. \cite{LPT}).

Furthermore, we have
\begin{lemma} For a $C^2$ function
$\psi{}:{}\mathbb{R}{}\to{}\mathbb{R}$, compactly supported on the
interval $[a,b]$, we have
$$
\supp\, \eta^\psi,\,\supp\, q^\psi{}\subset{}\left\{(\rho, m)=(\rho,
\rho u)\,:\, \rho^\theta{}+{}\u\geq a,\, u-\rho^\theta\leq b
\right\}.
$$
Furthermore, there exists a constant $C_\psi>0$ such that, for any
$\rho\ge 0$ and $u\in\R$, we have
\begin{enumerate}
\item[(i)] For $\gamma\in (1,3]$,
$$
|\eta^\psi(\rho, m)|+ |q^\psi(\rho, m)|{}\leq{}C_\psi \rho;
$$

\medskip
\item[(ii)] For $\gamma>3$,
$$
|\eta^\psi(\rho,m)|{}\leq{} C_\psi \rho,
\quad |q^\psi(\rho, m)|{}\leq{}C_\psi \rho\max\{1, \rho^\theta\};
$$

\medskip
\item[(iii)] If $\eta^\psi$ is considered as a function of
$(\rho, m), m=\rho u$, then
\[
|\eta^{\psi}_m(\rho, m)| +|\rho\, \eta^\psi_{mm}(\rho, m)|\le
C_\psi;
\]
and, if $\eta^\psi_m$ is considered as a function of $(\rho, u)$,
then
\[
|\eta^{\psi}_{mu}(\rho, \rho u)|
+|\rho^{1-\theta}\eta^\psi_{m\rho}(\rho, \rho u)|{}\leq{}C_\psi.
\]
\end{enumerate}

\end{lemma}
\begin{proof}
We first
notice that, if $(\rho,\u)$ is such that $\rho^\theta+\u{}<{}a$,
then $u+ \rho^\theta s{} <{}a$ for any $s\in[-1,1]$. Similarly, if
$u-\rho^\theta>b$, then $u+ s\rho^\theta{}>b$ for any $s\in[-1,1]$.

\medskip
For (i), since $\psi$ has compact support, it is clear from
\eqref{eta} that
$$
|\eta^\psi(\rho, m)|\le C_\psi \rho.
$$

When $\gamma=3$,
$$
q^\psi(\rho,m){} =\rho\int_{-1}^1(u+\rho s)\psi(u+\rho s)\,ds,
$$
which implies that $|q^\psi(\rho, m)|\le C_\psi \rho$ since $\psi$
has compact support.

When $\gamma<3$, we use the first formula in \eqref{q} to obtain
$$
|q^\psi(\rho, m)|\le C_\psi \rho^{2\theta\lambda+\theta} \le
C_\psi\rho.
$$

\medskip
For (ii), since $\psi$ has compact support, it is clear from the
formulas in \eqref{eta}--\eqref{q} that
$$
|\eta^\psi(\rho, m)|\le C_\psi \rho, \qquad |q^\psi(\rho, m)|\le
C_\psi \rho \max\{1, \rho^{\theta}\}.
$$

To prove (iv), we first notice that
$$
\eta^\psi_m(\rho, m) =\int\psi'(\frac{m}{\rho}+\rho^\theta
s)[1-s^2]_+^\lambda ds,
$$
which implies that $|\eta^\psi_m|\le C_\psi$. Furthermore, we have
$$
\eta^\psi_{mm}(\rho, m) =-\frac{1}{\rho}\int
\psi''(\frac{m}{\rho}+\rho^\theta s)[1-s^2]_+^{\lambda}ds,
$$
which yields that $|\rho\,\eta^\psi_{mm}(\rho, m)|\le C_\psi$.

When $\eta_m^\psi$ is regarded as a function of $(\rho, u)$,
$$
\eta^\psi_m(\rho, \rho u) =\int\psi'(u+\rho^\theta
s)[1-s^2]_+^\lambda ds.
$$
Then
\begin{equation}\label{2.16a}
\eta^\psi_{mu}(\rho, \rho u) =\int\psi''(u+\rho^\theta
s)[1-s^2]_+^\lambda ds,
\end{equation}
which leads to $|\eta^\psi_{mu}(\rho, \rho u)|\le C_\psi$; while
\begin{equation}\label{2.17a}
\eta^\psi_{m\rho}(\rho, \rho u)
=\theta\rho^{\theta-1}\int\psi''(u+\rho^\theta s)s [1-s^2]_+^\lambda
ds,
\end{equation}
which implies that $|\eta^\psi_{m\rho}(\rho, \rho u)|\le C_\psi
\rho^{\theta-1}$. This completes the proof.
\end{proof}

\begin{definition}\label{entropy-sol}
Let $(\rho_0, u_0)$ be given initial data with finite-energy with
respect to the end states $(\rho^\pm, u^\pm)$ at infinity, i.e.,
$E[\rho_0, u_0]\le E_0<\infty$. A pair of measurable functions
$(\rho, u): \R_+^2\to \R_+^2$ is called a finite-energy entropy
solution of the Cauchy problem \eqref{E-1} and
\eqref{Eq:Initial-Conditions} if the following holds:

\begin{enumerate}

\item[(i)] The total energy is bounded in time: There is a bounded function
$C(E,t)$, defined on $\mathbb{R}^+\times\mathbb{R}^+$ and continuous
in $t$ for each $E\in\mathbb{R}^+$, such that, for a.e. $t>0$,
$$
E[\rho, u](t)\le C(E_0,t);
$$

\item[(ii)] The entropy inequality:
$$
\eta^\psi(\rho, m)_t +q^\psi(\rho, m)_x \le 0
$$
is satisfied in the sense of distributions for the test function
$\psi(s)\in \{\pm 1, \pm s, s^2\}$;

\item[(iii)] The initial data $(\rho_0, u_0)$ are attained in the
sense of distributions.
\end{enumerate}
\end{definition}

The existence of entropy solutions in $L^\infty$ was established by
DiPerna \cite{DiPerna} for the case $\gamma=(N+2)/N, N\ge 5$ odd, by
Chen \cite{Chen} and Ding-Chen-Luo \cite{DCL} for the general case
$1<\gamma\le 5/3$ for usual gases, by Lions-Perthame-Tadmor
\cite{LPT} for the cases $\gamma\ge 3$, and by
Lions-Perthame-Souganidis \cite{LPS} for closing the gap $5/3<
\gamma<3$. The existence of finite-energy solutions was recently
established by LeFloch-Westdickenberg \cite{LW} for the case
$1<\gamma\le 5/3$ even for the spherically symmetric solutions. As a
corollary of Theorem 8.1 in this paper, the existence of
finite-energy entropy solutions is also established for the case
$\gamma>5/3$. Combining Propositions 6.2 and 7.2 with the estimates
in \cite{LW}, we also obtain the existence of finite-energy
solutions with spherical symmetry for the multidimensional Euler
equations for compressible, isentropic fluids for the case
$\gamma>5/3$.

\section{Uniform Estimates for the Solutions of the Navier-Stokes
Equations}

\medskip
Consider the Cauchy problem
\eqref{Eq:NS-1}--\eqref{Eq:Initial-Conditions} for the Navier-Stokes
equations in $\R_+^2:=[0,\infty)\times \R$.
Assume that $(\rho^\e(t,x), u^\e(t,x))$ are smooth solutions of
\eqref{Eq:NS-1}--\eqref{Eq:Initial-Conditions}, globally in time,
with $\rho^\e(t,x)\ge c_\e(t)$ for some $c_\e(t)>0$ for $t\ge 0$ and
$\lim_{x\to\pm\infty}(\rho^\e(t,x), u^\e(t,x))=(\rho^\pm,
u^\pm)$.

We now make several uniform estimates for the solutions
$(\rho^\e(t,x), u^\e(t,x))$ of
\eqref{Eq:NS-1}--\eqref{Eq:Initial-Conditions}, which are
independent of the physical viscosity coefficient $\e>0$. These
estimates are essential for establishing the convergence of
vanishing viscosity limit of solutions of the Cauchy problem
\eqref{Eq:NS-1}--\eqref{Eq:Initial-Conditions} for the Navier-Stokes
equations to a finite-energy entropy solution of the isentropic
Euler equations \eqref{E-1} with initial data
\eqref{Eq:Initial-Conditions}.

For simplification of notation, throughout this section, we denote
$\int=\int_{\R}$, $(\rho, u)=(\rho^\e, u^\e)$, and $C>0$ is a
universal constant independent of $\e$.

\smallskip
\subsection{Estimate I: Energy Estimate}

The total mechanical energy for \eqref{Eq:NS-1} in $\R$ introduced
in \eqref{def:energy} is equal to
$$
E[\rho, u](t)=\int\Big(\frac{1}{2}\rho(t,x)
|u(t,x)-\bar{u}(x)|^2+e^*(\rho(t,x),\bar{\rho}(x))\Big)\, dx,
$$
where
$e^*(\rho,\bar{\rho})=e(\rho)-e(\bar{\rho})-e'(\bar{\rho})(\rho-\bar{\rho})\ge
0
$ satisfies
$$
e^*(\bar{\rho},\bar{\rho})=e^*_\rho(\bar{\rho}, \bar{\rho})=0, \quad
e^*_{\rho\rho}(\rho,\rhob)=\frac{(\gamma-1)^2}{4}\rho^{\gamma-2}\ge
0  \qquad\mbox{for}\,\,\gamma>1.
$$
This implies that $e^*(\rho,\bar{\rho})$ is a convex function in
$\rho\ge 0$ that behaves like $\rho^\gamma$ for large $\rho$ and
like $(\rho-\bar{\rho})^2$ for $\rho$ close to $\bar{\rho}$. In
particular, for later use, we notice that there exists $C_0>0$ such that
\begin{equation}\label{e-1}
\rho(\rho^\theta-\bar{\rho}^\theta)^2\le C_0\, e^*(\rho, \bar{\rho})
\qquad \mbox{for}\,\, \rho\in [0,\infty),
\end{equation}
where $C_0$ is a continuous function of $\bar{\rho}$ and $\gamma$.

We start with the standard energy estimate.

\begin{lemma}[Energy Estimate]\label{lemma:3.1}
Let $E[\rho_0, u_0]\le E_0<\infty$, where $E_0>0$ is independent of
$\e$.
Then there exists $C=C(E_0,t,\bar{\rho},\bar{u})>0$, independent of
$\e$, such that
\begin{equation}
\label{B:1-Energy}
\sup_{\tau\in[0,t]} E[\rho,
u](\tau)
+{}\int_0^t\int\, \e|\u_x|^2\, dxd\tau {}\leq{}
C.
\end{equation}
\end{lemma}

This can be seen through the following direct calculation:
\begin{equation}
\label{lemma3.1:eq.1}
\frac{dE}{dt}{}={}\frac{d}{dt}\int\,\eta^*(\rho,m)\,dx
{}-{}\frac{d}{dt}\int\,\eta^*(\bar{\rho},\bar{m})\,dx{}
-{}\int\,\nabla\eta^*(\bar{\rho},\bar{m})\cdot(\rho_t,m_t)\,dx.
\end{equation}
Since $(\eta^*,q^*)$ is an entropy pair, we have
\[
\eta^{*}(\rho,m)_t{}+{}q^{*}(\rho,m)_x{}-{}\e \eta^*_m(\rho, m)
\,u_{xx}{}={}0,
\]
from which we conclude that
\begin{equation}
\label{lemma3.1:eq.2}
\frac{d}{dt}\int\,\eta^*(\rho,m)\,dx{}+{}\e\int\,|u_x|^2\,dx{}
={}q^*(\rho^-,m^-){}-{}q^*(\rho^+,m^+).
\end{equation}
The second integral in \eqref{lemma3.1:eq.1} depends only on $x$,
which implies that the second term on the right-hand side of
\eqref{lemma3.1:eq.1} vanishes. For the last integral, we employ
\eqref{Eq:NS-1} to obtain
\begin{eqnarray*}
\Big|
\int\,\nabla\eta^*(\bar{\rho},\bar{m})\cdot(\rho_t,m_t)\,dx\Big| &=&
\Big|-\int\,\nabla\eta^*(\bar{\rho},\bar{m})\cdot(m_x,(\rho u^2+p)_x-\e u_{xx})\,dx\Big| \\
&=&\Big|\int\,(\nabla\eta^*(\bar{\rho},\bar{m}))_x\cdot(m, \rho u^2+p -\e u_x)\,dx\Big| \\
&\leq&
\frac{\e}{2}\int\,|u_x|^2\,dx{}+\frac{1}{2}\int\,\rho|u-\bar{u}|^2\,dx\\
&& +C\Big(1+\int_{-L_0}^{L_0}(\rho+p)\,dx\Big).
\end{eqnarray*}
where we used that the compact support of $(\bar{\rho}_x,
\bar{u}_x)$ lies in the interval $[-L_0, L_0]$ for some $L_0>0$.
Since
$$
\int_{-L_0}^{L_0}(\rho+p)\,dx \le C\Big(1+\int_{-L_0}^{L_0}e^*(\rho,
\bar{\rho})\,dx\Big),
$$
we obtain
\begin{equation*}
\Big|
\int\,\nabla\eta^*(\bar{\rho},\bar{m})\cdot(\rho_t,m_t)\,dx\Big|{}\leq{}
\frac{\e}{2}\int\,|u_x|^2\,dx{} +C (E+1),
\end{equation*}
for some $C$ depending only on $(\gamma,\bar{\rho},\bar{u}).$
Combining this with \eqref{lemma3.1:eq.2}, we have
\[
\frac{dE}{dt}{}+{}\frac{\e}{2}\int\,|u_x|^2\,dx{}\leq{} CE{}+{}C.
\]
Then the lemma follows by Gronwall's inequality.

\medskip
\subsection{Estimate II: Space-Derivative Estimate for the Density}
We now develop an essential estimate for $\rho_x(t,x)$ involving the
$x$-derivative of the density, motivated by an argument in
\cite{Kanel}.

\begin{lemma}\label{lemma:3.2}
Let $(\rho_0,u_0)$ be such that
\[
\e^2 \int \frac{|\rho_{0,x}(x)|^2}{\rho_0(x)^3} dx \le E_1<\infty,
\]
where $E_1$ is independent of $\e$. Then there exists
$C=C(E_0,E_1,\bar{\rho},\bar{u},t)>0$ independent of $\e$ such that,
for any $t>0$,
\begin{equation}
\label{B:2-Energy} \e^2\int\frac{|\rho_x(t,x)|^2}{\rho(t,x)^3}dx{}
+{}\e\int_0^t\int\rho^{\gamma-3}|\rho_x|^2\, dxd\tau {}\leq{}
C.
\end{equation}
\end{lemma}

\begin{proof} Set $v=\frac{1}{\rho}$.
Then the first equation in \eqref{Eq:NS-1} can be written as
\[
v_t +u v_x=vu_x.
\]
Differentiating the above equation in $x$, we have
\begin{equation}
\label{3.1}
v_{xt} {}+{}(\u v_x)_x{}={}(v\u_x)_x.
\end{equation}
Then we multiply \eqref{3.1} by $2v_x$ to obtain
\[
(|v_x|^2)_t{}+{}\u (|v_x|^2)_x {}+{}2\u_x|v_x|^2{}={}2v_x(v\u_x)_x.
\]
Multiplying this by $\rho$ and using the equation of conservation of
mass yield
\[
(\rho|v_x|^2)_t{}+{}(\rho\u|v_x|^2)_x{}+{}2\rho\u_x|v_x|^2{}
={}2\rho v_x(v\u_x)_x,
\]
or
\begin{equation}
\label{3.1.1} (\rho|v_x|^2)_t{}+{}(\rho\u|v_x|^2)_x{}={}2v_x\u_{xx}.
\end{equation}

Using the second equation in \eqref{Eq:NS-1} and \eqref{3.1}, we
obtain
\begin{equation}
\label{3.2}
\begin{array}{ll}
2v_x\u_{xx}&=\frac{2}{\e}v_x\left( p_x {}+{}(\rho\u)_t{}+{}(\rho\u^2)_x\right)\medskip\\
&=\frac{2}{\e}v_xp_x{}+{}\frac{2}{\e}\Big( (\rho(\u-\bar{u})
v_x)_t{}-{}(\bar{u}(\ln\rho)_x)_t{}\underbrace{-{}\rho\u(v\u_x)_x{}+{}\rho\u(\u
v_x)_x{}+{}v_x(\rho\u^2)_x}_{J}\Big).
\end{array}
\end{equation}
By integration by parts, we have
\begin{eqnarray}
\int J \, dx &=&\int \big(v\u_x(\rho\u)_x{}-{}\u v_x(\rho\u)_x
{}+{}v_x(\u(\rho\u)_x{}+{}\rho\u\u_x)\big)\,dx
\nonumber\\
&=&\int\big(v\u_x(\rho\u)_x{}+{}\rho\u
v_x\u_x\big)\,dx=\int|\u_x|^2\, dx. \label{3.3}
\end{eqnarray}
Furthermore,
\begin{equation}
\label{3.4}
v_xp_x{}={}-\frac{(\gamma-1)^2}{4}\rho^{\gamma-3}|\rho_x|^2.
\end{equation}
Integrating \eqref{3.1.1} over $[0, t)\times \R$ and using the
calculations in \eqref{3.2}--\eqref{3.4}, we conclude
\begin{eqnarray}
&&\e^2\int\frac{|\rho_x(t,x)|^2}{\rho(t,x)^3}dx{}
+{}\frac{(\gamma-1)^2}{2}\e\int_0^t\int\rho^{\gamma-3}|\rho_x|^2dxd\tau
\nonumber\\
&&={}-2\e\int \frac{\rho_x(t,x)(u(t,x)-\bar{u}(x))}{\rho(t,x)}
dx{}+{}2\e\int\,\bar{u}(x)(\ln\rho)_x(t,x)\,dx
-{}2\e\int_0^t\int|\u_x|^2\,dxd\tau \nonumber\\
&&\quad +2\e\int \frac{\rho_{0,x}(x)(u_0(x)-\bar{u}(x))}{\rho_0(x)}
dx{}+{}2\e\int\,\bar{u}(x)(\ln\rho_0)_x(x)\,dx.
\label{lemma3.2:eq.0}
\end{eqnarray}
The first integral on the right-hand side is estimated by
\begin{eqnarray}
&&\frac{\e^2}{4}\int\,\frac{|\rho_x(t,x)|^2}{\rho(t,x)^3}dx{}
+{}8\int\,\rho(t,x)|u(t,x)-\bar{u}(x)|^2\,dx{}\nonumber\\
&&\leq{}
\frac{\e^2}{4}\int\,\frac{|\rho_x(t,x)|^2}{\rho(t,x)^3}dx{}+{}16E[\rho,u](t).
\label{lemma3.2:eq.3}
\end{eqnarray}
Similarly, the forth integral on the right-hand side is controlled
by
\begin{equation}
\frac{\e^2}{4}\int\,\frac{|\rho_{0,x}(x)|^2}{\rho_0(x)^3}dx{}+{}16E_0.
\label{lemma3.2:eq.3a}
\end{equation}
To estimate the second integral, we write
\[
2\e\int\,\bar{u}(\ln\rho)_x\,dx{}={}-2\e\int_{A_1}\bar{u}_x\ln\rho\,dx{}-{}
2\e\int_{A_2}\bar{u}_x\ln\rho\,dx +2\e\big(u^+\ln\rho^+-u^-\ln
\rho^-\big),
\]
where
\[
A_1{}={}\Big\{ x\,:\,
\rho(t,x)\leq\frac{\check{\rho}}{2}\Big\},\quad A_2{}={}A_1^c \qquad
\mbox{for}\,\,\,  \check{\rho}{}={}\min\{\rho^-,\rho^+\}.
\]
Since, on $A_2$, $|\ln\rho(t,x)|{}\leq{}C\rho(t,x)$ and $\bar{u}_x$
is compactly supported, we can obtain
\begin{equation}
\label{lemma3.2:eq.5} \left|
2\e\int_{A_2}\bar{u}_x\ln\rho\,dx\right| {}\leq{}C\Big(1
+\int\,e^*(\rho(t,x),\bar{\rho}(x))\,dx\Big).
\end{equation}

If the set $A_1$ is not
empty, then
\[
\left|2\e\int_{A_1}\bar{u}_x\ln\rho\,dx\right|{}\leq{}C\e\sup_{x\in
A_1}|\ln\rho(t,x)| {}\leq{}C\e\sup_{x\in
A_1}\frac{1}{\sqrt{\rho(t,x)}},
\]
and $A_1$ has finite measure, which can be estimated from
\eqref{B:1-Energy} by
\[
|A_1|{}\leq{}\frac{C}
{e^*(\frac{\check{\rho}}{2},\check{\rho})}{}=:{}d(t).
\]
In particular, for any $(t, x)$, there is a point $x_0(t,x)$ such
that $|x-x_0|\le d(t)$ and $\rho(t, x_0)=\frac{\check{\rho}}{2}$.
Then we have
\begin{eqnarray*}
\e\sup_{x\in A_1}\frac{1}{\sqrt{\rho(t,x)}} &\leq &\e\sup_{x\in
A_1}\Big|\frac{1}{\sqrt{\rho(t,x)}}{}
    -{}\frac{1}{\sqrt{\rho(t,x_0)}}\Big|{}+{}\frac{\e}{\sqrt{\check{\rho}/2}}
      \nonumber\\
&\leq&
\e\int_{x_0-d(t)}^{x_0+d(t)}\Big|\big(\frac{1}{\sqrt{\rho(t,x)}}\big)_x
\Big|\,dx{}+{}\frac{\e}{\sqrt{\check{\rho}/2}}\nonumber\\
&\leq&\Big(\frac{\e^2}{2}\int\,\frac{|\rho_x|^2}{\rho^3}\,dx\Big)^{1/2}\sqrt{d(t)}
 +{}\frac{\e}{\sqrt{\check{\rho}/2}}\nonumber\\
&\leq& \frac{\e^2}{4}\int\,\frac{|\rho_x|^2}{\rho^3}\,dx{}+{}C(t).
\end{eqnarray*}
Thus, we obtain
\[
2\e \Big| \int\,\bar{u}(\ln\rho)_x\,dx \Big|{}\leq{}
\frac{\e^2}{4}\int\,
\frac{|\rho_x|^2}{\rho^3}\,dx{}+{}C.
\]
Combining this with \eqref{lemma3.2:eq.3} in \eqref{lemma3.2:eq.0},
we obtain
\begin{eqnarray*}
&&\e^2\int\frac{|\rho_x(t,x)|^2}{\rho(t,x)^3}dx{}
+{}\frac{(\gamma-1)^2}{2}\e\int_0^t\int\rho^{\gamma-3}|\rho_x|^2dxd\tau \\
&&\leq \frac{\e^2}{2}\int\, \frac{|\rho_x(t,x)|^2}{\rho(t,x)^3}\,dx
{}+{}\e^2\int\frac{|\rho_{0,x}(x)|^2}{\rho_0(x)^3}
dx +{} C. \notag
\end{eqnarray*}
The estimate of the lemma then follows.
\end{proof}

\subsection{Estimate III: Higher Integrability}

We now make uniform estimates for higher integrability of the
solutions.

\begin{lemma}[Higher Integrability--I]\label{lemma:3.3}
Let $E[\rho_0, u_0]\le E_0<\infty$ for $E_0$ independent of $\e$.
Then,
for any $-\infty<a<b<\infty$ and all $t>0$, there exists $C=C(a, b,
E_0, \gamma, \bar{\rho},\bar{u}, t)>0$, independent of $\e>0$, such
that
$$
\int_0^t\int_a^b \rho(t,x)^{\gamma+1}\, dxd\tau\le C.
$$
\end{lemma}

\begin{proof}
Let $\omega(x)$ be an arbitrary smooth, compactly supported function
such that $0 \leq \omega(x){}\leq 1$. Multiplying the second
equation in \eqref{Eq:NS-1} by $\omega(x)$ and then integrating with
respect to the space variable over $(-\infty,x)$, we have
\[
\rho\u^2\omega{}+{}p\,\omega
{}={}\e\u_x\omega{}-{}\Big(\int_{-\infty}^x \rho\u\omega\,
dy\Big)_t{}+{}\int_{-\infty}^x\left( (\rho\u^2+
p)\omega_x{}-{}\e\u_x\omega_x\right)\,dy.
\]
Multiply this by $\rho\omega$ and use the first equation in
\eqref{Eq:NS-1} to obtain
\begin{eqnarray}
\rho p\,\omega^2
&=&-\rho^2\u^2\omega^2{}+{}\e\rho\u_x\,\omega^2{}-\Big(\rho\omega\int_{-\infty}^x
\rho\u\omega\, dy\Big)_t
   {}-{}(\rho\u)_x\,\omega\int_{-\infty}^x \rho\u\omega\, dy\nonumber\\
&&{}+{}\rho\omega\int_{-\infty}^x\left( (\rho\u^2+
p)\omega_x{}-{}\e\u_x\omega_x\right)\,dy
  \nonumber\\
&=& \e\rho\u_x\omega^2{}
-\Big(\rho\omega\int_{-\infty}^x\rho\u\omega\,dy\Big)_t
{}-{}\Big(\rho\u\omega\int_{-\infty}^x \rho\u\omega\, dy\Big)_x{}\nonumber\\
&&+{}\rho\u\omega_x\int_{-\infty}^x\rho\u\omega\,dy{}+{}
\rho\omega\int_{-\infty}^x\left( (\rho\u^2+
p)\omega_x{}-{}\e\u_x\omega_x\right)\,dy. \nonumber
\end{eqnarray}
Integrating the above equation over $(0,t)\times\R$, we have
\begin{eqnarray}
\int_0^t\int\rho
p\,\omega^2\,dyd\tau&=&\e\int_0^t\int\rho\u_x\omega^2\,dyd\tau{}
-{}\int
\rho\omega\Big(\int_{-\infty}^x\rho\u\omega\,dy\Big)\,dx\nonumber\\
&&+\int
\rho_0\omega\Big(\int_{-\infty}^x\rho_0\u_0\omega\,dy\Big)\,dx{}+{}r_1(t),
\label{Eq:H.I.-1}
\end{eqnarray}
where
\[
r_1(t){}={}\int_0^t\int\rho\u\omega_x\Big(\int_{-\infty}^x\rho\u\omega\,dy\Big)\,dxd\tau{}
+{}\int_0^t\int\rho\omega\Big(\int_{-\infty}^x\big((\rho\u^2+p)\omega_x
-\e\u_x\omega_x\big)\,dy\Big) \,dxd\tau.
\]

Note that, by the H\"{o}lder inequality, for any $\delta>0$,
\begin{eqnarray}
\e\int_0^t\int\rho\u_x\omega^2\,dxd\tau &\leq&
\frac{\e^2}{\delta}\int_0^t\int|\u_x|^2\,dxd\tau{}+{}\delta\int_0^t\int\rho^2\omega^4
\,dxd\tau\nonumber\\
&\leq& \frac{\e_0}{\delta}\e\int_0^t\int|\u_x|^2\,dxd\tau
{}+{}C\delta\int_0^t\int(1+\rho^{\gamma+1})\omega^2\,dxd\tau\nonumber\\
&\leq& C
 {}+{}C\delta\int_0^t\int\rho^{\gamma+1}\omega^2\,dxd\tau,
\label{Eq:H.I-1.1}
\end{eqnarray}
since $\e\in (0, \e_0]$. By Lemma \ref{lemma:3.1} and the H\"{o}lder
inequality, we have
\begin{eqnarray}
\left| \int_{-\infty}^x\rho\u\omega\,dy\right|
&\leq&\int_{\mbox{supp}\, \omega}|\rho\u|\,dy
{}\leq{}\Big(\int_{\mbox{supp}\, \omega}\rho\,dy\Big)^{1/2}
 \Big(\int_{\mbox{supp}\,\omega}\rho\u^2\,dy\Big)^{1/2}\nonumber\\
&\leq & C\Big(\int_{\mbox{supp}\,
\omega}\big(1+e^*(\rho,\bar{\rho})\big)\,dy \Big)^{1/2}
\Big(\int_{\mbox{supp}\, \omega}\rho\u^2\,dy\Big)^{1/2}\leq
C.\qquad
\end{eqnarray}
It follows then that
\begin{equation}
\label{Eq:H.I.-2} \left|
\int\rho\omega\Big(\int_{-\infty}^x\rho\u\omega\,dy\Big)\,dx\right|{}\leq{}C.
\end{equation}
Similarly, we have
\begin{eqnarray}
&&\left|\int_0^t\int\rho\u\omega_x\Big(\int_{-\infty}^x\rho\u\omega\,dy\Big)\,dxd\tau\right|
{}+{} \left|
\int_0^t\int\rho\omega\Big(\int_{-\infty}^x\big(\rho\u^2+p\big)\omega_x\,dy\Big)\,dxd\tau\right|
\nonumber\\
&&+{}\left|\int_0^t\int\rho\omega\Big(\int_{-\infty}^x\e\u_x\omega_x\,dy\Big)\,dxd\tau\right|
{}\leq{} C.
\label{Eq:H.I.-3}
\end{eqnarray}
Combining estimates \eqref{Eq:H.I-1.1}, \eqref{Eq:H.I.-2}, and
\eqref{Eq:H.I.-3} for the terms on the right-hand side of
\eqref{Eq:H.I.-1}, we obtain
\[
\int_0^t\int\rho^{\gamma+1}\omega^2\,dxd\tau{}
\leq{}C\delta\int_0^t\int\rho^{\gamma+1}\omega^2 dxdt{} +{}C.
\]
Choosing suitably small $\delta>0$, we conclude
\[
\int_0^t\int\rho^{\gamma+1}\omega^2\,dxd\tau
\leq{}C.
\]
\end{proof}

\begin{lemma}[Higher Integrability-II]\label{lemma:3.4}
Let $(\rho_0(x), u_0(x))$ satisfy, in addition to the conditions in
Lemmas {\rm 3.1}--{\rm 3.2},
\begin{equation}
\label{INITIAL:RHO-1} \int_{-\infty}^\infty
\rho_0(x)|\u_0(x)-\bar{u}(x)|\,dx{}\le M_0<{}\infty,
\end{equation}
where $M_0>0$ is a constant independent of $\e$.
Then, for any compact set $K\subset\mathbb{R}$ and $t>0$, there
exists $C>0$ independent of $\e$ such that
\begin{equation}
\int_0^t\int_K \Big(\rho|\u|^3{}+{}\rho^{\gamma+\theta}\big)\,
dxd\tau{}\leq{} C.
\end{equation}
\end{lemma}

\begin{proof}
Choose $\psi_\sh(w){}={}\frac{1}{2}w|w|$ in \eqref{eta}--\eqref{q}.
Then the corresponding weak entropy pair $(\eta^\sh,
q^\sh)=(\eta^{\psi_\sh}, q^{\psi_\sh})$ satisfies estimates
\eqref{2.0}--\eqref{2.1b}.

Note also that
\[
\eta^\sh(\rho,0){}={}\eta_\rho^\sh(\rho,0){}={}0,\qquad
q^\sh(\rho,0){}={}\frac{\theta}{2}\rho^{3\theta+1}\int|s|^3[1-s^2]_+^\lambda
ds>0,
\]
and
$$
\eta_m^\sh(\rho,0){}={}\alpha\rho^\theta{}\qquad\mbox{with}\,\,
\alpha:={}\int|s|[1-s^2]_+^\lambda ds.
$$
We also need the Taylor expansion of $\eta^\sh(\rho,m)$ at $m=0$ for
fixed $\rho$:
\begin{equation}
\label{Eq:Taylor} \eta^\sh(\rho,m){}={}\alpha \rho^\theta
m{}+{}r_2(\rho,m)
\end{equation}
with
\begin{equation}\label{3.19a}
|r_2(\rho,m)|{}\leq{}C \frac{m^2}{\rho}{}={}C\rho|\u|^2
\end{equation}
for some positive $C>0$. Finally, we introduce an entropy pair
$(\check{\eta},\check{q})$ by choosing the density function
$\psi(s){}={}\psi_\sh(s-u^-),$ where $u^-$ is the left end limit of
$u(t,x)$. Then
\[
\check{\eta}(\rho,m){}={}\eta^\sh(\rho,m-\rho u^-),\qquad
\check{q}(\rho,m){}={}q^\sh(\rho,m-\rho
u^-){}-{}u^-\eta^\sh(\rho,m-\rho u^-).
\]
Moreover, from \eqref{Eq:Taylor} and \eqref{3.19a}, we conclude
\begin{equation}
\label{Eq:Taylor1} \check{\eta}(\rho,m){}={}\alpha \rho^{\theta+1}
(u-u^-){}+{}r_2(\rho,\rho(u-u^-))
\end{equation}
with
\begin{equation}\label{3.19a1}
|r_2(\rho,\rho(u-u^-))|{}\leq{}C\rho|\u-u^-|^2.
\end{equation}

\medskip
Multiplying the first equation in \eqref{Eq:NS-1} by
$\check{\eta}_\rho$ and the second equation by $\check{\eta}_m$,
adding them together, and integrating the result over $(0, t)\times
(-\infty,x)$, we obtain
\begin{eqnarray}
 &&\int_{-\infty}^x\big(\check{\eta}(\rho, m){}
-{}\check{\eta}(\rho_0, m_0)\big)\,dy{}+{} \int_0^t q^\sh(\rho,
\rho(u-u^-))-u^-\eta^\sh(\rho,\rho(u-u^-))\,d\tau - t\tilde{q}
\nonumber\\
&& -{}\e\int_0^t\check{\eta}_m\u_x\,d\tau
 +\, \e\int_0^t\int_{-\infty}^x\check{\eta}_{mu}|u_x|^2\,dyd\tau {}
 +{}\e\int_0^t\int_{-\infty}^x\check{\eta}_{m\rho}\rho_x\u_x\,dyd\tau{}={}0,
 \label{2.3}
\end{eqnarray}
where $\tilde{q}=q^\sh(\rho^-,0)$. From the pointwise estimate
\eqref{2.1b} on $(\eta_{m\rho}^\sh,\,\eta_{m u}^\sh)$, which also
holds for $(\check{\eta}_{m\rho},\,\check{\eta}_{m u}),$ and Lemmas
\ref{lemma:3.1}--\ref{lemma:3.2}, we have
\begin{equation}
\label{2.4}
\Big|\e\int_0^t\int_{-\infty}^x\check{\eta}_{mu}|u_x|^2\,dyd\tau
\Big|{}\leq{} C,
\end{equation}
\begin{equation}
\label{2.5} \Big|
\e\int_0^t\int_{-\infty}^x\check{\eta}_{m\rho}\rho_x\u_x\,dyd\tau\Big|{}\leq{}
C.
\end{equation}
Using estimates \eqref{2.0}
and \eqref{2.4}--\eqref{2.5} in \eqref{2.3}, we obtain
\begin{eqnarray}
&&\int_0^t\int_K\left(\rho|\u-u^-|^3{}+{}\rho^{\gamma+\theta}\right)\,dxdt\nonumber\\
&&\leq  C(E_0,E_1,|K|,\bar{q},t){} +{}2\sup_{\tau\in[0,t]}
\Big|\int_K\Big(\int_{-\infty}^x\check{\eta}(\rho(y,\tau),
(\rho\u)(y,\tau))\,dy\Big)\,dx \Big|
 \nonumber\\
&&\quad +\sup|\bar{u}|\int_0^t\int_K|\eta^\sh(\rho,
\rho(u-u^-))|\,d\tau dx
  {}+{}\e C\int_0^t\int_K|\u||\u_x|\,dxd\tau\nonumber\\
&&\quad +{}\e C\int_0^t\int_K\rho^{\theta}|\u_x|\,dxd\tau.
\label{2.6}
\end{eqnarray}
Clearly, by the H\"{o}lder inequality,
\begin{eqnarray}
 \e\int_0^t\int_K\rho^\theta|\u_x|\, dxd\tau
\leq{}\e\int_0^t\int|\u_x|^2\,dxd\tau{}+{}
\e\int_0^t\int_K\rho^{\gamma-1}\,dxd\tau
\leq  C.
 \label{2.7}
\end{eqnarray}
Similarly,
\begin{eqnarray}
\e\int_0^t\int_K|\u||\u_x|\,dxd\tau &
\leq&\e\int_0^t\int_K|\u_x|^2\,dxd\tau{}+{}\e\int_0^t\int_K|\u|^2\,dxd\tau{}\nonumber\\
&\leq& C
{}+{}\e\int_0^t\int_K|\u|^2\,dxd\tau. \label{2.8}
\end{eqnarray}

Note from Lemma \ref{lemma:3.1} that there exists a nondecreasing
function $C(t)>0$ such that, for any $t>0$,
\[
\int_{\{\rho(t,x)\le \frac{\rhob}{2}\}}
e^*(\rho(t,\cdot),\rhob)\,dx{}\leq{}C(t),
\]
which implies that
\[
|\{x\,:\, \rho(t,x)\le \frac{\check{\rho}}{2}\}|{}
\leq{}\frac{C(t)}{e^*(\frac{\check{\rho}}{2},\check{\rho})}, \qquad
\check{\rho}{}={}\min\{\rho^-,\rho^+\}.
\]

Without loss of generality, we assume that $K$ contains the interval
$[a,b]$ of length
$\frac{2C(t)}{e^*(\frac{\check{\rho}}{2},\check{\rho})}$.
It follows then that, for any $t\ge 0$, there is a (measurable)
subset $A=A(t)\subset (a,b)$ of measure not less than
$\frac{C(t)}{e^*(\frac{\check{\rho}}{2},\check{\rho})}$ on which
$\rho(t,x)\geq\frac{\check{\rho}}{2}$.

Denote
\[
\u_A(t){}:={}\frac{1}{|A|}\int_A\u(t,x)\,dx.
\]
Then
\[
|\u(t,x)|{}\leq{} |\u_A(t)|{}+
{}
\int_K|\u_x|\,dx\qquad\mbox{for}\,\, x\in [a,b].
\]
We estimate
\begin{eqnarray*}
|u_A(t)|&\le &\frac{1}{|A|}\int_A |\u(t,x)|\,dx\\
&\leq& \frac{1}{|A|}\sqrt{\frac{2}{\check{\rho}}}\int_A\sqrt{\rho(t,x)}|\u(t,x)|\,dx\\
&\leq& \frac{1}{\sqrt{|A|}}\sqrt{\frac{2}{\check{\rho}}}\int
\rho(t,x)|\u(t,x)|^2\,dx\\
&\leq&
\sqrt{\frac{2C(t)e^*(\frac{\check{\rho}}{2},\check{\rho})}{\check{\rho}}}.
\end{eqnarray*}
Then
\begin{equation*}
 \e\int_0^t\int_K|\u|^2\,dxd\tau{}\leq{}
C\Big(\e\int_0^t\int|\u_x|^2\,dxd\tau{}
+{}\int_0^t|\u_A(\tau)|^2\,d\tau\Big){}\leq{}C,
\end{equation*}
and, from \eqref{2.8},
\begin{equation}\label{2.8.1}
\e\int_0^t\int_K|\u||\u_x|\,dxd\tau{}\leq{}C.
\end{equation}
Also, for the compact set $K$,
\begin{equation}
\label{2.8.1.1} \int_0^t\int_K\,|\eta^\sh(\rho,\rho(u-u^-))|\,d\tau
dx {}\leq{}C
\Big(1+\int_0^t\,E[\rho,u](\tau)\,d\tau\Big).
\end{equation}

Finally, we estimate the term
$\int_K\Big(\int_{-\infty}^x\check{\eta}(\rho,\rho\u)\,dy\Big)\,dx$.
Consider
\begin{eqnarray}
\Big|\int_{-\infty}^x\check{\eta}(\rho,\rho\u)\,dy\Big|
&=&\Big|\int_{-\infty}^x\big(\check{\eta}(\rho,\rho\u)-\alpha\rho^{\theta+1}(u-u^-)\big)dy\Big|{}
  +{}\Big|\int_{-\infty}^x\alpha\rho^{\theta+1}(u-u^-)\,dy\Big|\nonumber\\
&=&\Big|\int_{-\infty}^x r_2(\rho, \rho(u-u^-))\,dy\Big|
{}+{}\Big|\int_{-\infty}^x\alpha(\rho^\theta-(\rho^-)^\theta)\rho(u-u^-)\,dy\Big|\nonumber\\
&&\quad
+\Big|\alpha(\rho^-)^\theta\int_{-\infty}^x\rho(u-u^-)\,dy\Big|\nonumber\\
&\leq & C\Big( 1+ \int \big(\rho|u-u^-|^2+
e^*(\rho,\bar{\rho})\big)\,dx\Big){}
+{}\alpha(\rho^-)^\theta\int_{-\infty}^x\rho(u-u^-)\,dy\nonumber\\
&\leq&
C
{}+{}\alpha(\rho^-)^\theta\Big|\int_{-\infty}^x\rho(u-u^-)\,dy\Big|,
\label{Eq:some-1}
\end{eqnarray}
where we used \eqref{e-1}--\eqref{B:1-Energy},
\eqref{Eq:Taylor1}--\eqref{3.19a1} for $r_2(\rho,\rho(u-u^-)),$ and
the following inequality by using \eqref{e-1}: For $x\in K$,
\begin{eqnarray*}
\int_{-\infty}^x\rho(\rho^\theta-(\rho^-)^\theta)\,dx
  {}\leq{}C\int_{-\infty}^xe^*(\rho,\rho^-)\,dx
\leq{} C\Big(1{}+{}\int e^*(\rho,\bar{\rho})\,dx\Big).
\end{eqnarray*}

It remains to estimate
$\left|\int_{-\infty}^x\rho(u-u^-)\,dy\right|.$ For this, we
integrate equations in \eqref{Eq:NS-1} with respect to the
space-variable from $-\infty$ to $x$ and  the time-variable from $0$
to $t$:
\begin{eqnarray*}
&&\int_{-\infty}^x\rho(t,y)(u(t,y)-u^-)\,dy\\
&&={}\int_{-\infty}^x\rho_0(u_0-\bar{u})\,dy +\int_{-\infty}^x\rho_0(\bar{u}-u^-)\,dy\\
&&\quad -\int_0^t\big(\rho\u^2{}+{}p-p(\rho^-)
{}+{}u^-(\rho\u-\rho^-u^-)\big)\,d\tau{}+{}\e\int_0^t\u_x\,d\tau.
\end{eqnarray*}
Then, by a straightforward application of Lemma \ref{lemma:3.1}, we
obtain
\[
\int_K\Big|\int_{-\infty}^x\rho(t,y)(u(t,y)-u^-)\,dy
\Big|\,dx{}\leq{}C.
\]
Combining this with \eqref{Eq:some-1}, we have
\[
\int_K\Big|\int_{-\infty}^x\check{\eta}(\rho,\rho\u)(t,y)\,dy\Big|\,dx{}
\leq{}C.
\]

Using this, \eqref{2.7}, \eqref{2.8.1}, and \eqref{2.8.1.1} in
\eqref{2.6}, we conclude the proof.
\end{proof}

\begin{remark}
In the uniform estimate above, we require that the initial functions
$(\rho_0(x), u_0(x))$ satisfy

\begin{enumerate}
\item[(i)] $\rho_0(x)>0, \quad \int\rho_0(x)|u_0(x)-\bar{u}(x)|\,dx <\infty$;

\medskip
\item[(ii)] The total mechanical energy with respect to
$(\bar{\rho}, \bar{u})$ is finite:
$$
\int\Big(\frac{1}{2}\rho_0(x)|u_0(x)-\bar{u}(x)|^2
 +e^*(\rho_0(x),\bar{\rho}(x))\Big)dx=:E_0<\infty;
$$

\item[(iii)]
$\e^2\int \frac{|\rho_{0,x}(x)|^2}{\rho_0(x)^3}
\,dx\le E_1<\infty$.
\end{enumerate}

\medskip
Since our approach in dealing with the vanishing viscosity limit
below allows the vacuum, i.e. $\rho(t,x)\ge 0$, the initial
conditions (iii) and $\rho_0(x)>0$ can be removed by the standard
cutoff, $\max\{\rho_0(x), \e^{1/2}\}$, first and  mollification
$(\rho_0^\e(x), u^\e_0(x))\in C^\infty(\R)$ then, so that
$\rho_0^\e(x)\ge \e^{1/2}$ and
$$
\e^2\int
\frac{|\rho_{0,x}^\e(x)|^2}{\rho_0^\e(x)^3}
\,dx\le
E_1<\infty,
$$
for $E_1>0$ independent of $\e$.
\end{remark}

\section{$H^{-1}$--Compactness of the Weak Entropy Dissipation Measures}

In this section we establish the $H^{-1}$--compactness of entropy
dissipation measures for solutions to the Navier-Stokes equations
\eqref{Eq:NS-1} with initial data \eqref{Eq:Initial-Conditions} for
the weak entropy pairs generated by compactly supported $C^2$ test
functions $\psi$.

\begin{proposition}\label{prop:4.1}
Let $\psi: \R\to \R$ be any compactly supported $C^2$ function. Let
$(\eta^\psi, q^\psi)$ be a weak entropy pair generated by $\psi$.
Then, for the solutions $(\rho^\e, u^\e)$ with $m^\e=\rho^\e u^\e$
of the Navier-Stokes equations
\eqref{Eq:NS-1}--\eqref{Eq:Initial-Conditions}, the entropy
dissipation measures
\begin{equation}
\eta^{\psi}(\rho^\e, m^\e)_t{}+{}q^{\psi}(\rho^\e, m^\e)_x{}
\quad\mbox{are confined in a compact subset of }\,\,
H^{-1}_{loc}(\R_+^2).
\end{equation}
\end{proposition}

\begin{proof}
Multiplying the first equation in \eqref{Eq:NS-1} by
$\eta^{\psi}_\rho(\rho^\e, m^\e)$ and the second by
$\eta^{\psi}_m(\rho^\e, m^\e)$ and adding them up, we obtain
\begin{eqnarray}
 &&\eta^{\psi}(\rho^\e, m^\e)_t{}+{}q^{\psi}(\rho^\e, m^\e)_x\nonumber\\
&& ={}\e( \eta^{\psi}_m(\rho^\e, \rho^\e u^\e)\u^\e_x)_x{}-{}\e
\eta^{\psi}_{mu}(\rho^\e, \rho^\e u^\e)|\u^\e_x|^2{}-{} \e
\eta^{\psi}_{m\rho}(\rho^\e, \rho^\e u^\e)\rho^\e_x\u^\e_x,
\label{entropy-1}
\end{eqnarray}
where $\eta^{\psi}_{m\rho}(\rho, \rho
u)=\partial_\rho\big(\eta^{\psi}_{m}(\rho, \rho u)\big)$ and
$\eta^{\psi}_{mu}(\rho, \rho u)=\partial_u\big(\eta^{\psi}_{m}(\rho,
\rho u)\big)$.

Lemma 2.1 indicates that
$$
|\eta_{mu}^\psi(\rho^\e, \rho^\e
u^\e)|+|(\rho^\e)^{1-\theta}\eta_{m\rho}^\psi(\rho^\e, \rho^\e
u^\e)|\le C,
$$
where $C>0$ is independent of $\e$. Using this and the H\"{o}lder
inequality, we obtain that, for any $T\in (0, \infty)$,
\begin{eqnarray*}
&&\|\e \eta^{\psi}_{mu}(\rho^\e, \rho^\e u^\e)|\u^\e_x|^2{}+{} \e
\eta^{\psi}_{m\rho}(\rho^\e, \rho^\e u^\e)\rho^\e_x\u^\e_x\|_{L^1([0,T]\times\R)}\\
&&\le C_\psi\,\|(\sqrt{\e}u_x^\e,
\sqrt{\e}\rho^{\frac{\gamma-3}{2}}\rho_x^\e)\|_{L^2([0,T]\times\R)}
\le C.
\end{eqnarray*}
This yields that
\begin{equation}\label{h-1}
-{}\e \eta^{\psi}_{mu}(\rho^\e, \rho^\e u^\e)|\u^\e_x|^2{}-{} \e
\eta^{\psi}_{m\rho}(\rho^\e, \rho^\e u^\e)\rho^\e_x\u^\e_x
\qquad\mbox{is bounded in } L^1([0,T]\times\mathbb{R}),
\end{equation}
which implies its compactness in $W^{-1,q_1}_{loc}(\R_+^2),
1<q_1<2$.

Furthermore, since $|\eta_m^\psi(\rho^\e, \rho^\e u^\e)|\le C$, we
obtain
\begin{equation} \label{h-2}
\|\e \eta^{\psi}_m(\rho^\e, \rho^\e u^\e)
\u^\e_x\|_{L^2([0,T]\times\R)}\le
C\sqrt{\e}\|\sqrt{\e}u_x^\e\|_{L^2([0,T]\times\R)}\le C\sqrt{\e}\to
0\qquad \mbox{ as }\e\to 0.
\end{equation}
Combining \eqref{h-1} with \eqref{h-2} yields that
\begin{equation}\label{h-3}
\eta^{\psi}(\rho^\e, m^\e)_t{}+{}q^{\psi}(\rho^\e, m^\e)_x{}
\quad\mbox{are confined in a compact subset of }\,\,
W^{-1,q_1}_{loc}, 1<q_1<2.
\end{equation}

On the other hand, using the estimates in Lemma 2.1 (i)-(ii) and in
Lemmas \ref{lemma:3.3}--\ref{lemma:3.4}, we obtain that
$$
\eta^\psi(\rho^\e,m^\e),\,q^\psi(\rho^\e,m^\e) \qquad \mbox{are
uniformly bounded in }  L^{q_2}_{loc}(\R_+^2),
$$
for $q_2=\gamma+1>2$ when $\gamma\in(1, 3]$, and
$q_2=\frac{\gamma+\theta}{1+\theta}>2$ when $\gamma>3$. This implies
that, for some $q_2>2,$
\begin{equation}\label{h-4}
\eta^{\psi}(\rho^\e, m^\e)_t{}+{}q^{\psi}(\rho^\e, m^\e)_x{}
\qquad\mbox{are uniformly bounded in  }\,\, W^{-1,q_2}_{loc}.
\end{equation}

The interpolation compactness theorem (cf. \cite{Chen1,DCL})
indicates that, for $q_1>1$, $q_2\in(q_1, \infty]$, and $p\in [q_1,
q_2)$,
$$
\begin{array}{l}
(\hbox{compact set of}~~ W^{-1,q_1}_{loc}(\R_+^2))
\cap (\hbox{bounded set of}~~ W^{-1,q_2}_{loc}(\R_+^2))\\
\subset (\hbox{compact set of}~~ W^{-1,p}_{loc}(\R_+^2)),
\end{array}
$$
which is a generalization of Murat's lemma in \cite{Murat,Tartar}.

Combining this interpolation compactness theorem for $1<q_1<2,
q_2>2$, and $p=2$ with the facts in \eqref{h-3}--\eqref{h-4}, we
conclude the result.
\end{proof}

\section{Compensated Compactness and Measure-Valued Solutions}

In this section, we employ the estimates in Sections 3--4 to
construct the measure-valued solutions of the Cauchy problem
\eqref{Eq:NS-1}--\eqref{Eq:Initial-Conditions} for the Navier-Stokes
equations and show that the measure-valued solutions are confined by
the Tartar-Murat commutator relation for any two pairs of weak
entropy-entropy flux kernels via the method of compensated
compactness.

For convenience, we will work with measures defined on the phase
space:
$$
\H{}={}\{(\rho, u)\,:\, \rho>0\}.
$$
As in LeFloch-Westdickenberg
\cite{LW}, let $\bar{\HH}$ be a compactification of $\H$ such that
the space $C(\bar{\HH})$ is equivalent (isometrically isomorphic) to
the space
\begin{eqnarray*}
\bar{C}(\H){}={}\Big\{\phi\in C(\bar{\H})\, : {\begin{array}{ll} &
\phi(\rho,u) \mbox{ is constant on } \{\rho=0\}\,\,
\mbox{and the map} \\
&\mbox{$(\rho, u){}\to{}\lim_{s\to\infty}\phi(s\rho, su)$ belongs to
$C(\SPH\cap\bar{\H})$}
\end{array}}
\Big\},
\end{eqnarray*}
where $\mathbb{S}^1\subset\R^2$ is the unit circle. These spaces
allow to deal with the two difficulties of the problem when $\rho=0$
(vacuum) and when $\rho\gg 1$ in the large. As usual, we will not
distinguish between the functions in $\bar{C}(\H)$ and in
$C(\bar{\HH})$. The topology of $\bar{\HH}$ is the weak-star
topology induced by $C(\bar{\HH})$, which is separable and
metrizable. Note that the topology above does not distinguish points
in the compactification of the set $\{\rho=0\}$, that is, all points
in the vacuum are equivalent. Denote by $V$ the weak-star closure of
$\{\rho=0\}$ and define $\HH=\H\cup V$.

Following
Alberti-M\"{u}ller \cite{AM} (also see Ball \cite{Ball} and Tartar
\cite{Tartar}), we find that,
given any sequence of measurable functions $(\rho^\e,
u^\e){}:{}\R_+^2{}\to{}\bar{\HH}$, there exists a subsequence (still
labeled $(\rho^\e, u^\e)$) and a function
$$
\nu_{t,x}{}\in{}L^\infty_w\left(\R_+^2; \Prob(\bar{\HH}) \right)
$$
such that, for all $\phi\in C(\bar{\HH})$,
\begin{equation}
\label{Basic_Convergence}
\phi(\rho^\e(t,x), u^\e(t,x)){} \,\,\overset{*}{\rightharpoonup} \,
{} \int_{\bar{\HH}}\phi(\rho, u)\,d\nu_{t,x}(\rho, u)\qquad\,\,
\mbox{in } L^\infty\left(\R_+^2\right).
\end{equation}
The sequence of functions $(\rho^\e, u^\e)$ converges in measure to
$(\rho, m){}:{}\R_+^2{}\to{}\bar{\HH}$ if and only if
\[
\nu_{t,x}{}={}\delta_{(\rho(t,x), m(t,x))}\qquad a.e.{}\, (t,x).
\]

In what follows we will often abbreviate $\nu_{t,x}$ as $\nu$
implicitly assuming the dependence on $(t,x)$ when no confusion may
arise.

Let $B_R$ be a closed ball of radius $R$ centered at the origin. The
restriction of $\nu$ to $C(B_R\cap \bar{\H})$ can be identified with
a Radon (regular, Borel) measure $\nu_R{}\in{}C(B_R\cap
\bar{\H})^*$. By taking a sequence of radii, $R_n\to \infty$, we
obtain a probability measure $\nu$ on $\H$ such that, for any
$$
\phi{}\in{}C_0(\H){}={}\{\mbox{continuous functions, compactly
supported on}\,  \H\},
$$
we have
\begin{equation}
\int_{\H}\phi\,d\nu{}={} \la\nu,
\phi\ra_{\bar{C}(\H)\times\left(\bar{C}(\H)\right)^*},
\end{equation}
and
\begin{equation}
\label{Basic_Convergence-2} \phi(\rho^\e, u^\e){}
\overset{*}{\rightharpoonup} {} \int_{\H}\phi(\rho, u)\,d\nu\qquad
\mbox{ in } L^\infty\left(\R_+^2\right).
\end{equation}

We will often use later the same letter $\nu$ for an element of
$\left(\bar{C}(\H)\right)^*$, or $\left(C(\bar{\HH})\right)^*$, and
for its restriction (a Radon measure on $\H$) to
$\left(C_0(\H)\right)^*$, but it will be clear from the context
which one is used.

\medskip
Let $(\rho^\e,\u^\e)$ be the sequence of solutions of the
Navier-Stokes equations \eqref{Eq:NS-1} with initial data
\eqref{Eq:Initial-Conditions}. Let
$\nu=\nu_{t,x}$ be a Young measure corresponding to this sequence of
functions $(\rho^\e, u^\e)$.

In the following proposition (analogous to Proposition 2.3 in
\cite{LW}), we can extend the Young measure $\nu_{t,x}$ to a class
of test functions larger than $\bar{C}(\H)$.

\begin{proposition}
\label{PROP} The following statements hold:
\begin{itemize}

\item[(i)] For the Young measure $\nu_{t,x}$ introduced above,
\begin{equation}
\label{Decay_at_infinity} \int_\H
\big(\rho^{\gamma+1}{}+{}\rho|u|^3\big)\,d\nu_{t,x}\in L^1([0,
T]\times K).
\end{equation}

\item[(ii)]
Let $\phi(\rho, u)$ be a function such that
\begin{enumerate}

\medskip
\item[(a)] $\phi{}\in{} C_0(\bar{\H})$, i.e., continuous on $\bar{\H}$
   and zero on $\partial\H$;

\medskip
\item[(b)] $\supp\, \phi\subset\left\{ (\rho, u)\,:\, \rho^\theta
{}+{}\u\geq -c,\, u-\rho^\theta\leq c\right\}$ for some constant
$c>0$;

\medskip
\item[(c)] $|\phi(\rho, u)|{}\leq{} \rho^{\beta(\gamma+1)}$ for all $(\rho, u)$
with large $\rho$ and some $\beta\in(0,1)$.
\end{enumerate}

\medskip\noindent
Then $\phi$ is $\nu_{t,x}$--integrable and
\begin{equation}
\phi(\rho^\e, u^\e){}\,
{\rightharpoonup}{}\int_\H\phi\,d\nu_{t,x}\quad \mbox{ in }
L^1_{loc}\left(\R_+^2\right).
\end{equation}
\item[(iii)] For $\nu_{t,x}$ viewed as an element of $\left(C(\bar{\HH})\right)^*$,
\begin{equation}
\nu_{t,x}\left[\bar{\HH}\setminus\left(\H\cup V\right)\right]{}={}0,
\end{equation}
which means that $\nu_{t,x}$ is concentrated in $\H$ and/or on the
vacuum $V=\{\rho=0\}$.
\end{itemize}
\end{proposition}

\begin{proof}
To prove (i), we define a cut-off function $\omega_k(\rho,u)$ that
is nonnegative and continuous, equals $1$ on the box
\[
\left\{(\rho, u){}:{} \rho^\theta\in{}[\frac{1}{k},\,k], \, |u|\leq
k\right\}
\]
and equals to $0$ outside the box
\[
\left\{(\rho, u) {}:{} \rho^\theta{}\in{}[\frac{1}{2k},\,2k],\,
|u|\leq 2k\right\}.
\]
Then the functions
$\big((\rho^\e)^{\gamma+1}+\rho^\e|u^\e|^3\big)\omega_k(\rho^\e,
u^\e)$ are in $\bar{C}(\H)$ so that
\begin{eqnarray*}
&&\lim_{\e\to0}\int_{[0,T]\times
K}\big((\rho^\e)^{\gamma+1}+\rho^\e|u^\e|^3\big)
  \omega_k(\rho^\e, u^\e)\,dxdt\\
&&=\int_{[0,T]\times
K}\Big(\int_{\H}(\rho^{\gamma+1}+\rho|u|^3)\omega_k(\rho,u)\,d\nu_{t,x}\Big)\,dxdt,
\end{eqnarray*}
where $K$ is a compact subset of $\mathbb{R}$.
Note that, by Lemmas \ref{lemma:3.3}--\ref{lemma:3.4},
\[
 \int_{[0,T]\times K}\big((\rho^\e)^{\gamma+1}+\rho^\e|u^\e|^3\big)
  \omega_k(\rho^\e, u^\e)\,dxdt{}\leq{}C,
\]
where $C>0$ is independent of $\e>0$. By the monotone convergence
theorem,
\[
\lim_{k\to\infty}\int_{\H}(\rho^{\gamma+1}+\rho|u|^3)
\omega_k(\rho,u)\,d\nu{}\,={}
\int_{\H}(\rho^{\gamma+1}+\rho|u|^3)\,d\nu
\]
is a $(t,x)$--integrable function, which is finite a.e. $(t,x)\in
[0,T]\times K$:
\[
\int_{[0,T]\times
K}\Big(\int_{\H}(\rho^{\gamma+1}+\rho|u|^3)\,d\nu_{t,x}\Big)\,dxdt
<\infty.
\]

\medskip
To prove (ii), we define another cut-off function
$\hat{\omega}_k(\rho,u)$
such that $0\leq \hat{\omega}_k(\rho, u)\leq 1$,
$\hat{\omega}_k(\rho,u)$ is $1$ on the set
\[
\left\{ \frac{1}{k}{}\leq{}|(\rho^\theta, u)|{}\leq{}k,\, \arg
(\rho^\theta,
u){}\in{}[-\frac{\pi}{2}+\frac{1}{k},\frac{\pi}{2}-\frac{1}{k}]
\right\},
\]
and $\hat{\omega}_k(\rho,u)$ is $0$ outside the set
\[
\left\{ \frac{1}{2k}{}\leq{}|(\rho^\theta, u)|{}\leq{}2k,\, \arg (
\rho^\theta,
u){}\in{}[-\frac{\pi}{2}+\frac{1}{2k},\frac{\pi}{2}-\frac{1}{2k}]
\right\}.
\]
Note that, with $\phi(\rho, u)$ satisfying (ii)(a)-(c),
$\hat{\omega}_k(\rho, u)\phi(\rho, u){}\in{} \bar{C}(\H)$ and thus
$\la\nu_{t,x}, \hat{\omega}_k\,\phi\ra$ is well-defined for
$a.e.{}\,(t,x)$.

By the Lebesgue dominated convergence theorem and (i), it follows
that
\[
\lim_{k\to\infty}\int_\H\phi\, \hat{\omega}_k\,d\nu_{t,x}{}
={}\int_{\H}\phi\,d\nu_{t,x}\qquad\, \mbox{a.e. $(t,x)\in
[0,T]\times K$,}
\]
and
\[
\lim_{k\to\infty}\int_{[0,T]\times
K}\int_\H\phi\,\hat{\omega}_k\,d\nu_{t,x}\,dxdt{}={}\int_{[0,T]\times
K}\int_{\H}\phi\,d\nu_{t,x}\,dxdt.
\]
On the other hand, by definition of Young measures,
it implies that
\begin{equation}
\label{Eq:double_limit}
\lim_{k\to\infty}\lim_{\e\to0}\int_{[0,T]\times K}\la
\nu^\e_{t,x},\phi\,\hat{\omega}_k\ra\,dxdt{}={}\int_{[0,T]\times
K}\int_{\H}\phi\,d\nu_{t,x}\,dxdt.
\end{equation}

\medskip
\noindent{\bf Claim}. {\it
$\int_{[0,T]\times K}\la \nu^\e_{t,x}, \phi\,\hat{\omega}_k\ra\,dxdt
{}\to{}\int_{[0,T]\times K}\la\nu^\e_{t,x}, \phi\ra\,dxdt$ as
$k\to\infty$
uniformly for $\e{}\in{}[0,\e_0)$.}

\medskip
If this is true, then we can interchange the limits in
\eqref{Eq:double_limit} to obtain
\begin{eqnarray*}
\lim_{\e\to0}\int_{[0,T]\times K}\phi(\rho^\e(t,x), u^\e(t,x))\,dxdt
&=& \lim_{\e\to0}\int_{[0,T]\times K}\la\nu^\e_{t,x}, \phi\ra\,dxdt
\nonumber\\
&=&\lim_{\e\to0}\lim_{k\to\infty}\int_{[0,T]\times K}
\la\nu^\e_{t,x}, \phi\hat{\omega}_k\ra\,dxdt\nonumber\\
&=& \lim_{k\to\infty}\lim_{\e\to0}
\int_{[0,T]\times K}\la\nu^\e_{t,x}, \phi\hat{\omega}_k\ra\,dxdt\nonumber\\
&=&\lim_{k\to\infty}\int_{[0,T]\times K}\int_{\H}\phi\,\hat{\omega}_k\,d\nu_{t,x}\,dxdt\nonumber\\
&=&\int_{[0,T]\times K}\int_{\H}\phi\,d\nu_{t,x}\,dxdt,
\end{eqnarray*}
which is what we want.

\medskip
We now prove the claim. With $k_1<k_2$, consider
\[
\int_{[0,T]\times K}\la\nu_{t,x}^\e,
(\hat{\omega}_{k_1}-\hat{\omega}_{k_2})\phi\ra\,dxdt.
\]
Notice that
\[
\supp (\hat{\omega}_{k_1}-\hat{\omega}_{k_2}){}\subset{}
B_{\frac{1}{k_1}}(0){}\cup{}\left( \big(B_{k_1}(0)\big)^c \cap
B_{2k_2}(0)\right),
\]
\[
\sup_{B_{\frac{1}{k_1}}(0)}|\phi(\rho, u)|{}\le c_{k_1}\to 0 \qquad
\mbox{ as } k_1{}\to{}\infty,
\]
and, if $(\rho, u){}\in{}\supp\, \phi\cap \big(B_{k_1}(0)\big)^c$,
then
\[
\rho^\theta{}\geq{} \frac{k_1}{2}.
\]

Furthermore, by the Young's inequality, for any $\alpha>0$, there
exists $C(\beta,\alpha)>0$ such that
\[
|\phi(\rho, u)|{}\leq{} C(\beta,\alpha){}+{}\alpha \rho^{\gamma+1}.
\]
Thus we can estimate
\begin{eqnarray}
 &&\left| \int_{[0,T]\times K}\la\nu_{t,x}^\e,
(\hat{\omega}_{k_1}-\hat{\omega}_{k_2})\phi\ra\,dxdt\right|\nonumber\\
&&\leq{} T|K|\,c_{k_1}{}+{}C(\beta,\alpha)\Big|\left([0,T]\times
K\right) \cap \big\{(t,x)\,:\,
(\rho^\e)^\theta{}>{}\frac{k_1}{2}\big\}
\Big|\nonumber\\
&&\quad +{}\alpha\int_{[0,T]\times
K}|\rho^\e(t,x)|^{\gamma+1}\,dxdt. \label{EXT-1}
\end{eqnarray}
By the Chebyschev inequality,
\[
\Big|\left([0,T]\times K\right) \cap \big\{(t,x)\,:\,
(\rho^\e)^\theta{}>{}\frac{k_1}{2}\big\} \Big|
{}\leq{}\big(\frac{k_1}{2}\big)^{-\frac{\gamma+1}{\theta}} \int_{
[0,T]\times K}|\rho^\e(t,x)|^{\gamma+1}\,dxdt.
\]
Using the uniform estimate in Lemma \ref{lemma:3.3}, we deduce from
\eqref{EXT-1} that
\begin{eqnarray*}
&&\Big| \int_{[0,T]\times K}\la\nu_{t,x}^\e,
(\hat{\omega}_{k_1}-\hat{\omega}_{k_2})\phi\ra\,dxdt\Big|\nonumber\\
&&\leq {}T|K|\, c_{k_1}+
C(\beta,\alpha)\big(\frac{k_1}{2}\big)^{-\frac{\gamma+1}{\theta}}{}+{}C\Delta,
\end{eqnarray*}
where $C>0$ and $c_{k_1}$ are independent of $\e$, and $\alpha>0$ is
an arbitrary constant. The claim then follows.

\medskip
The result in (iii) follows directly from the uniform estimates for
$(\rho^\e, u^\e)$ in Lemmas \ref{lemma:3.3}--\ref{lemma:3.4} and
Proposition 5.1.
\end{proof}

For simplifying the notation, we denote the entropy kernel:
\[
\chi(\xi){}:={}[\rho^{2\theta}- (u-\xi)^2]_+^\lambda,
\]
and, for any function $f(\rho, u)$ with growth slower than $\rho
|u|^3+\rho^{\gamma+\max\{1, \theta\}}$,
$$
f(\rho^\e, u^\e) {}\,\, {\rightharpoonup}{}\,\, \overline{f(\rho,
u)}(t,x){}:={}\la \nu_{t,x}, f(\rho, u)\ra.
$$

\begin{proposition}\label{prop:5.2} Let $\nu_{t,x}$ be the Young measure
determined by the solutions of the Navier-Stokes equations
\eqref{Eq:NS-1} with initial data \eqref{Eq:Initial-Conditions}.
Then the Young measure $\nu_{t,x}$ is a measure-valued solution of
\eqref{Eq:NS-1}--\eqref{Eq:Initial-Conditions}: For the test
functions $\psi\in\{\pm 1, \pm s, s^2\}$.
\begin{equation}\label{mv-sol}
\langle \nu_{t,x}, \eta^\psi\rangle_t + \langle \nu_{t,x},
q^\psi\rangle_x \le 0,
\qquad \langle\nu_{t,x}, \eta^\psi\rangle(0,\cdot)=\eta^\psi(\rho_0,
\rho_0u_0),
\end{equation}
in the sense of distributions in $\R_+^2$. Furthermore, the
measure-valued solution $\nu_{t,x}$ is confined by the following
commutator relation: For a.e. $s_1, s_2\in \R$,
\begin{eqnarray}\label{commute}
\theta(s_2-s_1)\Big(\overline{\chi(s_1)\chi(s_2)}-
\overline{\chi(s_1)}\,\, \overline{\chi(s_2)}\Big)
=(1-\theta)\Big(\overline{u\chi(s_2)}\,\, \overline{\chi(s_1)}
-\overline{u\chi(s_1)}\, \,\overline{\chi(s_2)}\Big).
\end{eqnarray}
\end{proposition}

\begin{proof}
First, from \eqref{2.17a}, we find that, when $\psi\in \{\pm 1, \pm
s, s^2\}$,
$$
\eta^\psi_{m\rho}(\rho, \rho u)=0.
$$
Then we employ \eqref{entropy-1} and \eqref{2.16a} to obtain that
the solutions $(\rho^\e, u^\e)$ of
\eqref{Eq:NS-1}--\eqref{Eq:Initial-Conditions} satisfy
\begin{eqnarray}
&&\eta^{\psi}(\rho^\e, m^\e)_t{}+{}q^{\psi}(\rho^\e, m^\e)_x  \nonumber\\
&&\,\,={}\e( \eta^{\psi}_m(\rho^\e, m^\e)\u^\e_x)_x{}-{}\e \int
\psi''(\frac{m^\e}{\rho^\e}+ (\rho^\e)^\theta s)[1-s^2]_+^\lambda\,
ds \, |\u^\e_x|^2. \label{entropy-2}
\end{eqnarray}
When $\psi(s)\in\{\pm 1, \pm s, s^2\}$, $\psi''(s)\ge 0$, which
implies
\begin{eqnarray}
\eta^{\psi}(\rho^\e, m^\e)_t{}+{}q^{\psi}(\rho^\e, m^\e)_x\le
{}\e\big( \eta^{\psi}_m(\rho^\e, m^\e)\u^\e_x\big)_x.
\label{entropy-3}
\end{eqnarray}
Taking $\e\to 0$ in \eqref{entropy-3}, we conclude \eqref{mv-sol}.

\medskip Furthermore, combining Proposition 4.1 and the uniform
estimates in Lemmas 3.3--3.4 with the Div-Curl lemma (cf. Murat
\cite{Murat} and Tartar \cite{Tartar}), we deduce that, for any
$C^2$ compactly supported functions $\phi, \psi$, the quadratic
functions $\eta^\psi q^\phi -\eta^\phi q^\psi$ are weakly continuous
with respect to the weakly convergent physical viscosity sequence
$(\rho^\e, m^\e)\rightharpoonup (\rho, m)$:
\begin{equation}
\label{entropy-2} \eta^\psi(\rho^\e, m^\e) q^\phi(\rho^\e, m^\e)
-\eta^\phi(\rho^\e, m^\e) q^\psi(\rho^\e, m^\e)\, \rightharpoonup \,
\overline{\eta^\psi(\rho, m)}\,\, \overline{q^\phi(\rho, m)}
-\overline{\eta^\phi(\rho, m)}\,\, \overline{q^\psi(\rho, m)}
\end{equation}
in the sense of distributions in $[0,\infty)\times\R$.

In terms of the Young measure, \eqref{entropy-2} yields the
Tartar-Murat commutator relation:
\begin{equation}
\label{tartar} \overline{\eta_{\psi}q_{\phi}-\eta_{\phi}q_{\psi}}{}
={}\overline{\eta_{\psi}}\,\,\overline{q_{\phi}}{}
-{}\overline{\eta_{\phi}}\,\,\overline{q_{\psi}}.
\end{equation}
Thus, we have
\begin{eqnarray*}
&&\int \psi(s_1)\overline{\chi(s_1)}ds_1
\int\phi(s_2)\overline{(\theta s_2+(1-\theta)u)\chi(s_2)}ds_2\\
&& - \int \psi(s_2)\overline{\chi(s_2)}ds_2
\int\phi(s_1)\overline{(\theta
s_1+(1-\theta)u)\chi(s_1)}ds_1\\
&&=\int \psi(s_1)\phi(s_2)\overline{\chi(s_1)(\theta
s_2+(1-\theta)u)\chi(s_2)}ds_1ds_2\\
&&\quad - \int \psi(s_1)\phi(s_2)\overline{\chi(s_1)(\theta
s_1+(1-\theta)u)\chi(s_1)\chi(s_2)}ds_1ds_2,
\end{eqnarray*}
which holds for arbitrary functions $\psi$ and $\phi$. This yields
\begin{eqnarray*}
\overline{\chi(s_1)}\,\, \overline{(\theta
s_2+(1-\theta)u)\chi(s_2)}-
\overline{\chi(s_2)}\,\,\overline{(\theta
s_1+(1-\theta)u)\chi(s_1)}
=\theta(s_2-s_1)\overline{\chi(s_1)\chi(s_2)},
\end{eqnarray*}
which implies \eqref{commute}.
\end{proof}

\smallskip
\section{Reduction of the Measure-Valued Solutions for $\gamma\in
(3,\infty)$}

\smallskip
In this section, we prove that any connected component of the
support of the measure-valued solution  $\nu=\nu_{t,x}$ must be
bounded for a.e. $(t,x)\in\R_+^2$.

\begin{lemma}
\label{Lemma:6.1} Let $\gamma > 3$. Then
$$
\overline{\chi(\xi)}{}\in{}L^1_{loc}(\R_+^2; L^p(\mathbb{R}))\qquad
\mbox{for} \,\, 1\le p<\frac{\gamma-1}{\gamma-3}.
$$
\end{lemma}

This can be seen by the following direct calculation: For any
$K\Subset \R$ and $T\in (0,\infty)$,
\begin{eqnarray*}
\int_{[0, T]\times K}\|\overline{\chi(\xi)}\|_{L^p}dxdt
&\leq&\int_{[0,T]\times K}\int_\H\,\Big(
\int[\rho^{2\theta}-(u-\xi)^2]_+^{p\lambda}\,d\xi\Big)^{1/p}
d\nu_{t,x}\, dxdt\\
&=&\int_{[0, T]\times K}\int_\H \rho^{\frac{\theta}{p}(2\lambda p
+1)}\Big(\int_{-1}^1
(1-\tau^2)^{p\lambda}d\tau\Big)^{1/p}d\nu_{t,x}\,dxdt \\
&\le & C\int_{[0,T]\times K}\int_\H \max\{1, \rho\} \, d\nu_{t,x}\,
dxdt<{}\infty,
\end{eqnarray*}
if $\frac{\theta}{p}(2\lambda p +1){}>{}0$ and $p\lambda{}>-1$,
which hold if $1\le p<\frac{\gamma-1}{\gamma-3}$.

\medskip

Let $A$ be the open set defined as
\[
A:={}\cup \{ (u-\rho^\theta, u+\rho^\theta)\,:\,(\rho,u){}\in{}
\supp\, \nu \},
\]
and let $J$ be any connected component of $A$.

\begin{proposition} When $\gamma>3$,
$J$ is bounded. That is, any connected component of the support of
the measure-valued solution $\nu$ is bounded.
\end{proposition}

\begin{proof}
Note that
\[
\supp\, \chi(\xi){}={}\{ (\rho,u)\,:\, u-\rho^\theta\le s\le
u+\rho^\theta\}.
\]
By definition of $J$, $\chi(\xi)>0$ for $a.e.\,\,\xi\in J$.

{}From \eqref{commute}, we obtain that,
if $\chi(\xi_1)\,\chi(\xi_2)\not=0$, then
\begin{equation}
\label{LPT-29} \frac{1-\theta}{\theta}\frac{1}{\xi_2-\xi_1}
\Big(\frac{\overline{u\chi(\xi_2)}}{\overline{\chi(\xi_2)}}{}
-{}\frac{\overline{u\chi(\xi_1)}}{\overline{\chi(\xi_1)}}\Big) =
\frac{\overline{\chi(\xi_1)\chi(\xi_2)}}
{\overline{\chi(\xi_1)}\,\overline{\chi(\xi_2)}}{} -1.
\end{equation}

Taking the limits $\xi_1,\xi_2{}\to{}\xi$ in \eqref{LPT-29} (cf.
\cite{LPT}, pp. 426), we conclude
that
\begin{equation}\label{monotone-0}
\frac{1-\theta}{\theta}\frac{\partial}{\partial
\xi}\Big(\frac{\overline{u\chi(\xi)}}{\overline{\chi(\xi)}}\Big){}
={}\frac{\overline{\chi^2(\xi)}}{\big(\overline{\chi(\xi)}\big)^2}-1{}\geq
0.
\end{equation}
This implies that
the function
\begin{equation}\label{monotone-1}
\frac{1-\theta}{\theta}\frac{\overline{u\chi(\xi)}}{\overline{\chi(\xi)}}
\quad\mbox{is non-decreasing on $J$.}
\end{equation}
Consequently, from \eqref{LPT-29}, we obtain
\begin{equation}
\label{Formula-1}
\frac{\overline{\chi(\xi_1)\chi(\xi_2)}}{\overline{\chi(\xi_1)}}{}\geq{}\overline{\chi(\xi_2)}
\qquad\, \mbox{a.e.}\,\,\,\,\xi_1,\xi_2\in J,\quad \xi_1<\xi_2.
\end{equation}

On the contrary, suppose now that $J$ is unbounded from below, that
is, $\inf\{s{}:{}s\in J\}=-\infty$.

We fix $M_0>0$ such that $M_0+1{}\in{} J$ and restrict
$\xi_2{}\in{}(M_0,M_0+1)$. We will take $\xi_1\leq-2|M_0|$.
For such $\xi_1$,
\begin{equation}
\label{big_xi_1} |M_0-\xi_1|>\frac{|\xi_1|}{2}.
\end{equation}
If $(\rho,u)\in \supp\,\chi(\xi_2)\cap \supp\,\chi(\xi_1)$, then, by
the above assumptions on $\xi_1$ and $M_0$, we have
\[
\rho^\theta -u+\xi_2=\rho^\theta-u+\xi_1+ (s_2-s_1)\geq s_2-s_1\ge
M_0-s_1>\frac{|\xi_1|}{2}.
\]
Since $\gamma>3$, i.e. $\lambda<0$, it follows that
\begin{eqnarray}
\int\chi(\xi_1)\chi(\xi_2)\, d\nu &=&\int
\chi(\xi_1)[\rho^\theta-u+s_2]_+^\lambda[\rho^\theta+u-s_2]_+^\lambda
\,d\nu\nonumber\\
&\leq& 2^{-\lambda}|\xi_1|^\lambda\int_{\supp\,\chi(\xi_2)}
\chi(\xi_1)[\rho^\theta+u-\xi_2]_+^\lambda\,d\nu. \label{Eq-1}
\end{eqnarray}
We integrate \eqref{Eq-1} in $\xi_2$ over the interval
$(M_0,\,M_0+1)$ to obtain
\begin{eqnarray}
&&\int_{M_0}^{M_0+1}\int\chi(\xi_1)\chi(\xi_2)\, d\nu d\xi_2\nonumber\\
&&\leq 2^{-\lambda}|\xi_1|^\lambda
\int_{M_0}^{M_0+1}\int_{\supp\,\chi(\xi_2)}
\chi(\xi_1) [\rho^\theta+u-\xi_2]^\lambda_+ d\nu d\xi_2\nonumber\\
&&=2^{-\lambda}|s_1|^\lambda
\int\chi(\xi_1)\Big(\int_{(M_0,M_0+1)\cap(u-\rho^\theta,u+\rho^\theta)}
[\rho^\theta+u-\xi_2]_+^\lambda\,d\xi_2 \Big)d\nu. \nonumber\\
\label{Eq-2}
\end{eqnarray}

We now consider the integral in the parentheses in \eqref{Eq-2}.

When $\rho^\theta+u\geq M_0+2$, then
$\rho^\theta+u-\xi_2{}\geq{}M_0+2-(M_0+1){}={}1$ and
\[
\int_{(M_0,M_0+1)\cap(u-\rho^\theta,u+\rho^\theta)}
[\rho^\theta+u-\xi_2]_+^\lambda\,d\xi_2{}\leq{}1,
\]
since $\lambda<0$.

When $\rho^\theta+u<M_0+2$, then
\begin{eqnarray*}
\int_{(M_0,M_0+1)\cap(u-\rho^\theta,u+\rho^\theta)}[\rho^\theta+u-\xi_2]_+^\lambda\,d\xi_2
&\leq&\int_{M_0}^{M_0+1}[\rho^\theta+u-\xi_2]_+^\lambda\,d\xi_2\\
&\leq&
\frac{1}{1+\lambda}[\rho^\theta+u-M_0]_+^{1+\lambda}\\
&\leq& \frac{1}{1+\lambda} 2^{1+\lambda},
\end{eqnarray*}
since $1+\lambda{}>{}0$.

Combining the two observations above into \eqref{Eq-2}, we find that
there exists  $C=C(\lambda)>0$ such that
\begin{equation}
\int_{M_0}^{M_0+1}\int\chi(\xi_1)\chi(\xi_2)d\nu
d\xi_2{}\leq{}C(\lambda)|\xi_1|^{\lambda}\overline{\chi(\xi_1)}.
\end{equation}
Combining this with \eqref{Formula-1}, we obtain
\[
C(\lambda)|\xi_1|^\lambda{}\geq{}\int_{M_0}^{M_0+1}\overline{\chi(\xi_2)}d\xi_2\equiv
C(M_0,\lambda)>0.
\]
Since $\lambda<0$ and $|\xi_1|$ can be chosen arbitrary large,
we arrive at a contradiction.

The case when $J$ is unbounded from above can be treated similarly.
\end{proof}

\medskip
With this proposition, a simple argument (cf. \cite{LPT}, Lemma 6)
 implies that $\nu$ is reduced to a Dirac mass on the set
$\{\rho>0 \}$ or is supported completely in the vacuum
$V=\{\rho=0\}$ for the case $\gamma>3$. This can be seen as follows:
Let $J=(s_-, s_+)$ be the open connected component. Then the values
$(\rho, u)$ such that $\chi(s)>0$ in an interval $(s_+-\varepsilon,
s_+)$ satisfy
$$
u+\rho^\theta\ge s_+-\varepsilon.
$$
Since $s_-\le u-\rho^\theta$ for these $(\rho, u)$ values, we have
\begin{equation}\label{monotone-2}
\lim_{s\to s_+}\frac{\overline{u\chi(s)}}{\overline{\chi(s)}} \ge
\min\{u\,:\, (\rho, u)\in \supp\, \nu, u+\rho^\theta =s_+\} \ge
\frac{s_++s_-}{2}.
\end{equation}
Similarly, we have
\begin{equation}\label{monotone-3}
\lim_{s\to s_-}\frac{\overline{u\chi(s)}}{\overline{\chi(s)}} \le
\frac{s_++s_-}{2}.
\end{equation}
Combining \eqref{monotone-2}--\eqref{monotone-3} with
\eqref{monotone-1}, we conclude that
$\frac{\overline{u\chi(s)}}{\overline{\chi(s)}}$ is constant, which
implies from \eqref{monotone-0} that
$$
\overline{\chi(s)^2}={\overline{\chi(s)}}^2.
$$
Since $\nu_{t,x}$ is a probability measure,
$$
\langle \nu_{t,x}, (\chi(s)-\langle \nu_{t,x},
\chi(s)\rangle)^2\rangle=0 \qquad \mbox{for any}\,\, s\in \R,
$$
which yields
$$
\supp\, \nu_{t,x}\subset \{\chi(s)=\langle \nu_{t,x},
\chi(s)\rangle\} \qquad\mbox{for any}\,\, s\in \R.
$$
This arrives at the conclusion. That is, in the phase coordinates
$(\rho, m), m=\rho u$,
$$
\nu_{t,x}=\delta_{(\rho(t,x), m(t,x))}
$$
for some
$(\rho(t,x), m(t,x))$.

\bigskip

When $\gamma=3$, then $\theta=1$ and the commutator relation
\eqref{commute} reads
\[
\overline{\chi(s_1)\chi(s_2)}{}={}\overline{\chi(s_1)}\,\,\overline{\chi(s_2)},
\]
which implies $\overline{\chi(s)^2}={\overline{\chi(s)}}^2$ by
taking $s_1=s_2$. This again implies that $
\nu_{t,x}=\delta_{(\rho(t,x), m(t,x))} $ for some $(\rho(t,x),
m(t,x))$.

\begin{proposition}\label{thm:6.1}
When $\gamma\ge 3$, the measure-valued solution $\nu_{t,x}$ is a
Dirac mass in the phase coordinates $(\rho, m)$:
$$
\nu_{t,x}=\delta_{(\rho(t,x), m(t,x))}.
$$
\end{proposition}

\medskip
\section{Reduction of the Measure-Valued Solutions for $\gamma\in
(1, 3)$}

In this section, we directly prove that any connected component of
the support of the measure-valued solution $\nu=\nu_{t,x}$ is
bounded.

\begin{lemma}
\label{Lemma:7.1} When $\gamma{}\in{}(1,3)$, $\overline{\chi(\xi)}$
is a continuous and weakly differentiable function for which
\[
\frac{\partial}{\partial
\xi}\overline{\chi(\xi)}{}\in{}L^1_{loc}(\R_+^2; L^1(\mathbb{R})).
\]
\end{lemma}

This can been seen as follows: We compute
\[
\partial_\xi\chi(\xi){}={}2\lambda(u-s)[\rho^{2\theta}-(u-\xi)^2]_+^{\lambda-1},
\]
and
\begin{eqnarray*}
\int|\partial_\xi\overline{\chi(\xi)}|d\xi &=&2\lambda\int
\Big(\int_{u-\rho^\theta}^u(u-s)[\rho^{2\theta}-(u-\xi)^2]_+^{\lambda-1}\,d\xi\Big)d\nu_{t,x}\\
&&+2\lambda\int\Big(\int^{u+\rho^\theta}_u(s-u)[\rho^{2\theta}-(u-\xi)^2]_+^{\lambda-1}\,d\xi\Big)
d\nu_{t,x}\\
&\le &C(\lambda)\int \rho^{2\theta\lambda}\,d\nu_{t,x}{}\in
L^1_{loc}(\R_+^2),
\end{eqnarray*}
since $0<2\theta\lambda{}\leq{}\gamma+1$ and by using  Proposition
\ref{PROP}(i).

\bigskip
Let $A$ be the open set defined as
\[
A :={}\cup \{ (u-\rho^\theta,\rho^\theta+u)\,:\,(\rho,u){}\in{}
\supp \nu \}
\]
and let $J$ be any connected component of $A$.

\begin{proposition} When $\gamma\in (1,3)$,
$J$ is bounded.
\end{proposition}

\begin{proof} We divide the proof into three steps.

{\it Step 1}. On the contrary, suppose as before that $J$ is
unbounded from below and let $M_0{}={}\sup\{s{}:{}s\in
J\}\in(-\infty,\infty]$.

Let $\xi_1, \xi_2, \xi_3\in (-\infty, M_0)$ with
$\xi_1<\xi_2<\xi_3$. From equation \eqref{commute}, it can be
derived that
\begin{eqnarray}
&&(\xi_2-\xi_1)\frac{\overline{\chi(\xi_1)\chi(\xi_2)}}{\overline{\chi(\xi_1)}}
{}+{}(\xi_3-\xi_2)\frac{\overline{\chi(\xi_3)\chi(\xi_2)}}{\overline{\chi(\xi_3)}}\nonumber\\
&&=(\xi_3-\xi_1)\overline{\chi(\xi_2)}\frac{\overline{\chi(\xi_1)\chi(\xi_3)}}
{\overline{\chi(\xi_1)}\,\overline{\chi(\xi_3)}}. \label{LPT-50}
\end{eqnarray}
Differentiating this equation in $\xi_2$ and dividing by
$\xi_3-\xi_1$, we obtain
\begin{eqnarray}
&&\frac{\xi_2-\xi_1}{\xi_3-\xi_1}\frac{\overline{\chi(\xi_1)\chi'(\xi_2)}}{\overline{\chi(\xi_1)}}
{}+{}\frac{\xi_3-\xi_2}{\xi_3-\xi_1}\frac{\overline{\chi(\xi_3)\chi'(\xi_2)}}{\overline{\chi(\xi_3)}}
{}+{}\frac{1}{\xi_3-\xi_1}\frac{\overline{\chi(\xi_1)\chi(\xi_2)}}{\overline{\chi(\xi_1)}}
{}-{}\frac{1}{\xi_3-\xi_1}\frac{\overline{\chi(\xi_3)\chi(\xi_2)}}{\overline{\chi(\xi_3)}}\nonumber\\
&&{}={}
\overline{\chi'(\xi_2)}\frac{\overline{\chi(\xi_1)\chi(\xi_3)}}
{\overline{\chi(\xi_1)}\,\overline{\chi(\xi_3)}}. \label{49}
\end{eqnarray}

Our strategy is to take $\xi_1\to-\infty$ and show that the
left-hand side of \eqref{49} has a smaller order than the right-hand
side, which arrives at a contradiction.

\medskip
{\it Step 2}. {\it Claim: $
 \overline{\chi(\xi)}{}\to{}0$  as $\xi\to-\infty$ and $\xi\to M_0$}.

\medskip
If $M_0<\infty$, then the result follows by the definition of $J$
and the fact that $\overline{\chi(s)}$ is continuous (which follows
from Lemma \ref{Lemma:7.1}).

\medskip
We now  show that $\overline{\chi(s)}\to 0$ as $|s|\to\infty$ for
$M_0=\infty$.

Using Lemma 3.3 and Young's inequality, we have
\begin{eqnarray*}
\overline{\chi(s)}&=&\int_\H[\rho^{2\theta}-(u-s)^2]_+^\lambda
d\nu{}\leq{}
\int_{\H\,\cap\,\supp\,\chi(s)}\rho^{2\theta\lambda}d\nu\\
&\leq&\e^{2\lambda}{}+{}\int_{\H\,\cap\{\rho^\theta\geq\e\}\,\cap\,\supp\,\chi(s)}
\big(C(\delta){}+{}\delta\rho^{\gamma+1}\big)\,d\nu\\
&\leq&\e^{2\lambda}{}+{}\delta
C{}+{}C(\delta)\nu\big(\{\rho^\theta\geq
R\}\cup\{\rho^\theta\geq\e,\,|u|\geq R\}),
\end{eqnarray*}
where $\e$ and $\delta$ are positive constants (to be taken small)
and $C(\delta)$ is some constant depending on the negative powers of
$\delta$ and $R:=\frac{|s|}{4}$. Then, by Chebyshev's inequality and
Proposition \ref{PROP}(i),  we conclude
\[
\nu(\{\rho^\theta\geq
R\})\leq{}\frac{\int\rho^{\gamma+1}d\nu}{R^{\frac{\gamma+1}{\theta}}}{}
\leq{}\frac{M}{R^{\frac{\gamma+1}{\theta}}},
\]
\[
\nu(\{\rho^\theta\geq\e,\,|u|\geq
R\})\leq{}\frac{\int\rho|u|^3d\nu}{\e^{1/\theta}R^3}{}\leq{}\frac{N}{\e^{1/\theta}R^3},
\]
where $M$ and $N$ are the constants depending only on $(t,x)$. Thus,
choosing first $\delta$ small, then $\e$ small, and finally $R$
(i.e. $|s|$) large, we can make $\overline{\chi(s)}$ as small as we
want.

\medskip
{\it Step 3}. Now we prove Proposition 7.1. Since
$\overline{\chi(\xi)}{}\geq{}0$ is not identically zero and
$$
\overline{\chi(\xi)}{}\to{}0 \qquad\mbox{ as } \xi\to\inf J,\,\sup
J,
$$
there exists $\xi_2$ such that
\begin{equation}
\label{50-1/2} \overline{\chi'(\xi_2)}{}>{}0,\qquad
\overline{\chi(\xi_2)}{}>{}0.
\end{equation}
Moreover, following the same argument for \eqref{Formula-1} from
\eqref{commute}, we still have
\begin{equation}
\label{51}
\frac{\overline{\chi(\xi_1)\chi(\xi_3)}}{\overline{\chi(\xi_1)}
\,\,\overline{\chi(\xi_3)}}{}\geq{}1 \qquad \mbox{for any}\,\, s_1,
s_3\in J.
\end{equation}
Let $\xi_3>\xi_2$ be points such that $\overline{\chi(\xi_3)}>0$ and
let $\xi_1\to-\infty$.  Then, from \eqref{LPT-50}, we conclude
\begin{equation}
\label{53}
\frac{\overline{\chi(\xi_1)\chi(\xi_2)}}{\overline{\chi(\xi_1)}}
{}={}
\overline{\chi(\xi_2)} \frac{\overline{\chi(\xi_1)\chi(\xi_3)}}
{\overline{\chi(\xi_1)}\,\,\overline{\chi(\xi_3)}}{}+{}o(1) \qquad
\mbox{as}\,\, \xi_1\to-\infty.
\end{equation}
From \eqref{49}, by throwing away the negative terms, we obtain
\begin{equation}
\label{54} \overline{\chi'(\xi_2)}
\frac{\overline{\chi(\xi_1)\chi(\xi_3)}}{\overline{\chi(\xi_1)}\,\,\overline{\chi(\xi_3)}}{}
\leq{}
\frac{\overline{\chi(\xi_1)[\chi'(\xi_2)]_+}}{\overline{\chi(\xi_1)}}{}+{}
\frac{1}{\xi_3-\xi_1}\frac{\overline{\chi(\xi_1)\chi(\xi_2)}}{\overline{\chi(\xi_1)}}
+{}o(1),
\end{equation}
where $[w]_+$ stands for the nonnegative part of $w$. For
$(\rho,u)\in \supp [\chi'(s)]_+$, consider
\begin{eqnarray*}
[\chi'(s)]_+&=& 2\lambda[\rho^\theta -s
+u]_+^{\lambda-1}[\rho^\theta+s-u]_+^{\lambda-1}[u-s]_+\\
&=&2\lambda[\rho^\theta -s
+u]_+^{\lambda}[\rho^\theta+s-u]_+^{\lambda}\frac{1}{[\rho^\theta+s-u]_+}
\frac{[u-s]_+}{[u-s+\rho^\theta]_+}\\
&\leq& 2\lambda[\rho^\theta -s
+u]_+^{\lambda}[\rho^\theta+s-u]_+^{\lambda}\frac{1}{[\rho^\theta+s-u]_+}.
\end{eqnarray*}
Note that, if $(\rho,u)\in\supp\,\chi(s_1)$, then $\rho^\theta\geq
u-s_1$. If, in addition $(\rho,u)\in \supp\,\chi(s)$ with $s>s_1$,
then
\[
\rho^\theta +s-u{}\geq{} s-s_1.
\]
Thus, we have
\begin{equation} \label{55}
[\chi'(s)]_+{}\leq{}\frac{2\lambda}{s - s_1}[\rho^\theta -s
+u]_+^{\lambda}[\rho^\theta+s-u]_+^{\lambda}{}={}\frac{2\lambda}{s -
s_1}\chi(s),
\end{equation}
when $(\rho,u)\in\supp\,\chi(s_1)\cap\supp\,\chi(s)$ for $s_1<s$.
Setting $s=\xi_2$ and using \eqref{55} in \eqref{54}, we obtain
\begin{equation}
\label{56} \overline{\chi'(\xi_2)}
\frac{\overline{\chi(\xi_1)\chi(\xi_3)}}{\overline{\chi(\xi_1)}\,\,\overline{\chi(\xi_3)}}
{}\leq{} \Big(\frac{2\lambda}{\xi_2-\xi_1}+
\frac{1}{\xi_3-\xi_1}\Big)
\frac{\overline{\chi(\xi_1)\chi(\xi_2)}}{\overline{\chi(\xi_1)}}
{}+{}o(1).
\end{equation}
{}From this, recalling \eqref{53}, we obtain
\begin{equation}
\Big(
\overline{\chi'(\xi_2)}{}-{}\frac{2\lambda\overline{\chi(s_2)}}{\xi_2-\xi_1}{}
-{}\frac{\overline{\chi(s_2)}}{\xi_3-\xi_1}\Big)
\frac{\overline{\chi(\xi_1)\chi(\xi_3)}}{\overline{\chi(\xi_1)}\,\,\overline{\chi(\xi_3)}}{}
\leq{}o(1).
\end{equation}
Because of \eqref{50-1/2} and \eqref{51}, the last inequality is a
contradiction when $s_1\to-\infty$.
This completes the proof.
\end{proof}

Then, by the well-known result, see \cite{DiPerna,Chen,DCL,LPS}, the
measure-valued solution $\nu$ reduced to a delta function in the
phase coordinates $(\rho, m)$.

\begin{proposition}\label{thm:7.1}
When $\gamma\in (1, 3)$, the measure-valued solution $\nu_{t,x}$ is
a Dirac mass in the phase coordinates $(\rho, m)$:
$$
\nu_{t,x}=\delta_{(\rho(t,x), m(t,x))}.
$$
\end{proposition}

\begin{remark} The above proof provides another way to establish the
reduction of measure-value solutions, which simplifies the  proof by
LeFloch-Westdickenberg \cite{LW}.
\end{remark}

\section{Vanishing Viscosity Limit of the Navier-Stokes Equations
     to the Euler Equations with Finite-Energy Initial Data}

\medskip
Consider the Cauchy problem
\eqref{Eq:NS-1}--\eqref{Eq:Initial-Conditions} for the Navier-Stokes
equations in $\R_+^2:=\R\times [0,\infty)$. Hoff's theorem in
\cite{Hoff} (also see Kanel \cite{Kanel} for the case of the same
end states) indicates that, when the initial functions $(\rho_0(x),
u_0(x))$ are smooth with the lower bounded density $\rho_0(x)\ge
c^\e_0>0$ for $x\in\R$ and
$$
\lim_{x\to\pm\infty}(\rho_0(x), u_0(x))=(\rho^\pm, u^\pm),
$$
then there exists a unique smooth solution $(\rho^\e(t,x),
u^\e(t,x))$, globally in time, with $\rho^\e(t,x)\ge c_\e(t)$ for
some $c_\e(t)>0$ for $t\ge 0$ and $\lim_{x\to\pm\infty}(\rho^\e(t,x),
u^\e(t,x))=(\rho^\pm, u^\pm)$.

Combining the uniform estimates  and Remark 3.1 in Section 3 and the
compactness of weak entropy dissipation measures in $H^{-1}_{loc}$
in Section 4 with the compensated compactness argument in Section 5
and the reduction of the measure-valued solution $\nu_{t,x}$ in
Sections 6--7, we conclude the following main theorem of this paper.

\begin{theorem}\label{main-theorem}
Let the initial functions $(\rho_0^\e, u_0^\e)$ be smooth and
satisfy the following conditions: There exist $E_0, E_1, M_0>0$,
independent of $\e$, and $c_0^\e>0$ such that
\begin{enumerate}
\item[(i)] $\rho_0^\e(x)\ge c_0^\e>0, \quad \int\rho_0^\e(x)|u_0^\e(x)-\bar{u}(x)|\,dx \le M_0
<\infty$;

\medskip
\item[(ii)] The total mechanical energy with respect to
$(\bar{\rho}, \bar{u})$ is finite:
$$
\int\Big(\frac{1}{2}\rho_0^\e(x)|u_0^\e(x)-\bar{u}(x)|^2
 +e^*(\rho_0^\e(x),\bar{\rho}(x))\Big)dx\le E_0<\infty;
$$

\item[(iii)]
$\e^2\int \frac{|\rho_{0,x}^\e(x)|^2}{\rho_0^\e(x)^3}
\,dx\le E_1<\infty$;

\item[(iv)]
$(\rho^\e_0(x), \rho^\e_0(x)u^\e_0(x))\to (\rho_0(x),
\rho_0(x)u_0(x))$ in the sense of distributions as $\e\to 0$, with
$\rho_0(x)\ge 0$ a.e.,
\end{enumerate}
where $(\bar{\rho}(x),\bar{u}(x))$ is some pair of smooth monotone
functions satisfying $(\bar{\rho}(x), \bar{u}(x))=(\rho^\pm, u^\pm)$
when $\pm x\ge L_0$ for some large $L_0>0$. Let $(\rho^\e, m^\e),
m^\e=\rho^\e u^\e$, be the solution of the Cauchy problem
\eqref{Eq:NS-1}--\eqref{Eq:Initial-Conditions} for the Navier-Stokes
equations with initial data $(\rho^\e_0(x), u^\e_0(x))$ for each
fixed $\e>0$. Then, when $\e\to 0$, there exists a subsequence of
$(\rho^\e, m^\e)$ that converges almost everywhere to a
finite-energy entropy solution $(\rho, m)$ to the Cauchy problem
\eqref{E-1} and \eqref{Eq:Initial-Conditions} with initial data
$(\rho_0(x), \rho_0(x)u_0(x))$ for the isentropic Euler equations
with $\gamma>1$.
\end{theorem}

\bigskip
\bigskip
{\bf Acknowledgments.} Gui-Qiang Chen's research was supported in
part by the National Science Foundation under Grants DMS-0935967,
DMS-0807551, and DMS-0505473, the Natural Science Foundation of
China under Grant NSFC-10728101, and the Royal Society--Wolfson
Research Merit Award (UK). This paper was written as part of the
International Research Program on Nonlinear Partial Differential
Equations at the Centre for Advanced Study at the Norwegian Academy
of Science and Letters in Oslo during the Academic Year 2008--09.


\bigskip
\bigskip
\begin{thebibliography}{xxx}

\bibitem{AM} G. Alberti and S. M\"{u}ller, {\sl A new approach to
variational problems with multiple scales}, {\it Comm. Pure Appl.
Math.} {\bf 54} (2001), 761--825.

\bibitem{Ball} J. Ball, {\sl A version of the fundamental theorem of Young
measures}, In: {\it PDEs and Continuum Models of Phase Transitions},
pp. 207--215, Eds. Rascle, Serre, and Slemrod, Lecture Notes of
Physics, {\bf 344}, Springer-Verlag, 1989.

\bibitem{BB} S. Bianchini and A. Bressan,
{\sl Vanishing viscosity solutions of nonlinear hyperbolic systems},
{\it Ann. of Math. (2)},  {\bf 161} (2005), 223--342.

\bibitem{Chen} G.-Q. Chen, {\sl Convergence of the Lax-Friedrichs
scheme for isentropic gas dynamics (III)}, {\it Acta Math. Sci.}
{\bf 6B} (1986), 75--120 (in English); {\bf 8A} (1988), 243--276 (in
Chinese).

\bibitem{Chen1} G.-Q. Chen, {\sl The compensated compactness method
and the system of isentropic gas dynamics}, Lecture Notes, Preprint
MSRI-00527-91, Berkeley, October {\bf 1990}.

\bibitem{Chen2} G.-Q. Chen, {\sl Remarks on R. J. DiPerna's paper:
``Convergence of the viscosity method for isentropic gas dynamics''
[Comm. Math. Phys. 91 (1983),  1--30]}, {\it Proc. Amer. Math. Soc.}
{\bf 125} (1997), 2981--2986.

\bibitem {CLone} G.-Q. Chen and Ph. G. LeFloch, {\sl Compressible
Euler equations with general pressure law}, {\it Arch. Rational
Mech. Anal.} {\bf 153} (2000), 221--259; {\sl Existence theory for
the isentropic Euler equations}, {\it Arch. Rational Mech. Anal.}
{\bf 166} (2003), 81--98.

\bibitem{Dafermos} C.~M. Dafermos,
{\sl Hyperbolic Conservation Laws in Continuum Physics},
Springer-Verlag: Berlin, 2000.

\bibitem{Ding} X. Ding, {\sl On a lemma of DiPerna and Chen},
{\it Acta Math. Sci.} {\bf 26B} (2006), 188--192.

\bibitem{DCL}  X. Ding, G.-Q. Chen, and P. Luo, {\sl Convergence of
the Lax-Friedrichs scheme for the isentropic gas dynamics (I)-(II)},
{\it Acta Math. Sci.} {\bf 5B} (1985), 483-500, 501--540 (in
English); {\bf 7A} (1987), 467-480; {\bf 8A} (1989), 61--94 (in
Chinese); {\sl Convergence of the fractional step Lax-Friedrichs
scheme and Godunov scheme for the isentropic system of gas
dynamics}, {\it Comm. Math. Phys.} {\bf 121} (1989), 63--84.

\bibitem{DiPerna} R.~J. DiPerna, {\sl Convergence of the viscosity
method for isentropic gas dynamics}, {\it Commun. Math. Phys.} {\bf
91} (1983), 1--30.

\bibitem{DiPerna-g} R.~J. DiPerna, {\sl Convergence of approximate solutions
to conservation laws}, {\it Arch. Rational Mech. Anal.} {\bf 82}
(1983), 27--70.

\bibitem{Gi} D. Gilbarg, {\sl The existence and limit behavior of the
one-dimensional shock layer}, {\it Amer. J. Math.} {\bf 73} (1951),
256--274.


\bibitem{GMWZ} O. Gu\`{e}s, G. M\'{e}tivier, M. Williams, and K. Zumbrun,
{\sl  Navier-Stokes regularization of multidimensional Euler
shocks}, {\it Ann. Sci. \'{E}cole Norm. Sup. (4)}, {\bf 39} (2006),
75--175.

\bibitem{Hugoniot} H. Hugoniot,
{\sl Sur la propagation du movement dans les corps et
ep\'{e}cialement dans les gaz parfaits}, {\it J. Ecole
Polytechnique}, {\bf 58} (1889), 1--125.


\bibitem{Hoff} D. Hoff,
{\sl Global solutions of the equations of one-dimensional,
compressible flow with large data and forces, and with differing end
states}, {\it Z. Angew. Math. Phys.} {\bf 49} (1998), 774--785.

\bibitem{HL} D. Hoff and T.-P. Liu, {\sl The inviscid limit for the
Navier-Stokes equations of compressible, isentropic flow with shock
data},  {\it Indiana Univ. Math. J.} {\bf 38} (1989), 861--915.


\bibitem{Kanel} Ya. Kanel, {\sl On a model system of equations of
one-dimensional gas motion}, {\it Diff. Urav.} {\bf 4} (1968),
721--734.

\bibitem{Laxone} P.~D. Lax, {\sl Shock wave and entropy}, In:
{\it Contributions to Functional Analysis}, ed. E.A. Zarantonello,
pp. 603--634, Academic Press: New York, 1971.


\bibitem{LW}
Ph. LeFloch and M. Westdickenberg, {\sl Finite energy solutions to
the isentropic Euler equations with geometric effects}, {\it J.
Math. Pures Appl.} {\bf 88} (2007), 386--429.


\bibitem{LPS} Lions P.-L., Perthame B., and Souganidis P.E.,
{\sl Existence and stability of entropy solutions for the hyperbolic
systems of isentropic gas dynamics in Eulerian and Lagrangian
coordinates}, {\it Comm. Pure Appl. Math.} {\bf 49} (1996),
599--638.

\bibitem{LPT} P.-L. Lions, B. Perthame, and E. Tadmor,
{\sl Kinetic formulation of the isentorpic gas
dynamics and p-systems,\/} Commun. Math. Phys. {\bf 163} (1994),
415--431.

\bibitem{Mor} C.  Morawetz,  {\sl An alternative proof of
DiPerna's theorem}, {\it Comm. Pure Appl. Math.} {\bf 44} (1991),
1081--1090.

\bibitem{Murat} F. Murat, {\sl Compacit\'e par compensation},
{\it Ann. Scuola Norm. Sup. Pisa Sci. Fis. Mat.} {\bf 5} (1978),
489--507.

\bibitem{PerthameTzavaras} B. Perthame and A. Tzavaras,
{\sl Kinetic formulation for systems of two conservation laws and
elastodynamics}, {\it Arch. Ration. Mech. Anal.} {\bf 155} (2000),
1--48.

\bibitem{Rankine} W.~J.~M. Rankine, {\sl On the thermodynamic theory
of waves of finite longitudinal disturbance}, {\it Phi. Trans. Royal
Soc. London}, {\bf 1960} (1870), 277-288.

\bibitem{Rayleigh}
Lord Rayleigh (J.~W. Strutt), {\sl Aerial plane waves of finite
amplitude}, {\it Proc. Royal Soc. London}, {\bf 84A} (1910),
247--284.

\bibitem{Serre-g} D. Serre,
{\sl La compacit\'{e} par compensation pour les syst\`{e}mes
hyperboliques non lin\'{e}aires de deux \`{e}quations \`{a} une
dimension d'espace}, {\it J. Math. Pures Appl. (9)}, {\bf 65}
(1986), 423--468.


\bibitem{SS}
D. Serre and J. W. Shearer, {\sl Convergence with physical viscosity
for nonlinear elasticity}, Preprint, 1994 (unpublished).

\bibitem{Stokes}
G.~G. Stokes, {\sl On a difficulty in the theory of sound}, {\it
Philos. Magazine}, {\bf 33} (1848), 349--356.

\bibitem{Tartar} L. Tartar, {\sl Compensated compactness and
applications to partial differential equations}, In: {\it Research
Notes in Mathematics, Nonlinear Analysis and Mechanics},
Herriot-Watt Symposium, Vol. {\bf 4}, R. J. Knops ed., Pitman Press,
1979.

\end{thebibliography}
\end{document}